\documentclass[11pt]{amsart}
\usepackage{geometry}     
\usepackage[parfill]{parskip}    
\usepackage{graphicx}
\usepackage{amssymb}
\usepackage{epstopdf}
\usepackage{xcolor}
\usepackage{fullpage}
\usepackage{tikz-cd}
\usepackage{mathrsfs}
 \usepackage{comment} 
\usepackage[all, 2cell]{xy}
\usepackage{hyperref}
\usepackage{marginnote}
\usepackage[normalem]{ulem}

\theoremstyle{plain}
\newtheorem{theorem}{Theorem}[section]
\newtheorem*{theorem*}{Theorem}

\newtheorem{lemma}[theorem]{Lemma}
\newtheorem{corollary}[theorem]{Corollary}
\newtheorem{proposition}[theorem]{Proposition}

\newtheorem{remark}[theorem]{Remark}

\newtheorem{setup}[theorem]{Setup}

\theoremstyle{definition}
\newtheorem{example}[theorem]{Example}

\newtheorem{definition}[theorem]{Definition}

\newtheorem*{theorem-non}{Theorem}

\newcommand{\A}{\mathcal{A}}

\newcommand{\C}{\mathrm{C}}
\newcommand{\HH}{\mathrm{H}}
\newcommand{\Ee}{\mathcal{E}}
\newcommand{\D}{\mathsf{D}}
\newcommand{\E}{\mathcal{E}}
\newcommand{\F}{\mathcal{F}}

\newcommand{\K}{\mathsf{K}}

\newcommand{\Tt}{\mathbf{T}}
\newcommand{\T}{\mathcal{T}}
\newcommand{\U}{\mathcal{U}}
\newcommand{\V}{\mathcal{V}}
\newcommand{\W}{\mathcal{W}}
\newcommand{\X}{\mathcal{X}}
\newcommand{\Y}{\mathcal{Y}}

\newcommand{\Zn}{\mathbb{Z}}
\DeclareMathOperator{\Hom}{\mathrm{Hom}}
\DeclareMathOperator{\Ext}{\mathrm{Ext}}
\DeclareMathOperator{\dgP}{\mathrm{dgP}}
\DeclareMathOperator{\dgI}{\mathrm{dgI}}
\DeclareMathOperator{\All}{\mathrm{All}}
\DeclareMathOperator{\Ac}{\mathrm{Ac}}
\DeclareMathOperator{\Ko}{\mathrm{K}}
\newcommand{\Mod}{\mathrm{Mod}\text{-}}
\newcommand{\Proj}{\mathrm{Proj}\text{-}}
\newcommand{\Inj}{\mathrm{Inj}\text{-}}
\newcommand{\inj}{\mathrm{inj}\text{-}}
\newcommand{\proj}{\mathrm{proj}\text{-}}
\newcommand{\fmod}{\mathrm{mod}\text{-}}
\newcommand{\Ch}{\mathrm{Ch}}
\newcommand{\hclim}{\mathrm{hocolim}}
\newcommand{\clim}{\mathrm{colim}}
\newcommand{\hlim}{\mathrm{holim}}
\newcommand{\thick}{\mathrm{thick}}
\newcommand{\tria}{\mathrm{tria}}
\newcommand{\Add}{\mathrm{Add}}
\newcommand{\add}{\mathrm{add}}
\newcommand{\Prod}{\mathrm{Prod}}
\DeclareMathOperator{\Coker}{\mathrm{Coker}}
\newcommand{\Fac}{\mathrm{Fac}}
\newcommand{\Sub}{\mathrm{Sub}}

\title{Limits and colimits in silting theory with applications to the wall and chamber structure of an algebra}
\author{Rosanna Laking}
\address[Laking]{Dipartimento di Informatica - Settore di Matematica
Università degli Studi di Verona,
Strada le Grazie 15,
37134 Verona, Italy}
\email{rosanna.laking@univr.it}
\author{Alexandra Zvonareva}
\address[Zvonareva]{Institute of Mathematics of the Czech Academy of Sciences, Žitná 25, 115 67 Prague,
Czech Republic}
\email{zvonareva@math.cas.cz}

\begin{document}
\maketitle

\begin{abstract}
    In this paper we consider a family of nested t-structures given by silting objects and  construct a silting object corresponding to the intersection of aisles of these t-structures as a homotopy colimit. The dual construction for the cosilting case is given as a homotopy limit. The results are applied to construct two-term large silting objects corresponding to the numerical torsion pairs and the limiting walls in the wall and chamber structure of the real Grothendieck group of a finite dimensional algebra. In particular, in case the algebra is tame we can describe any numerical torsion pair in this way by combining our results with results of Plamondon and Yurikusa. 
\end{abstract}

\section{Introduction}

The wall and chamber structure and the $g$-vector fan of an algebra connect the areas of cluster algebras and Bridgeland stability conditions with silting theory and $\tau$-tilting theory. The wall and chamber structure of the real Grothendieck group of a finite-dimensional algebra $A$ was introduced by Bridgeland \cite{Bridgeland} using King's stability. It gives a partition of $\Ko_0(\proj A) \otimes_\mathbb{Z} \mathbb{R} \simeq \mathbb{R}^l$, where $l$ is the rank of the Grothendieck group of $A$. The walls of the wall and chamber structure correspond to semistable modules in the sense of King, while the chambers are the open connected components of $\mathbb{R}^l$ with the closure of all walls removed. 
The union of the walls form the support of the Hall algebra scattering diagram, which provides a connection to cluster algebras and mirror symmetry. 
Bridgeland shows that the wall and
chamber structure
governs the wall crossing phenomena in an open subset of the space of Bridgeland stability conditions, a complex manifold associated to the category $\D^b(\fmod A)$ and parameterising certain homological information including bounded t-structures satisfying some technical conditions. 

The $g$-vector fan of an algebra was introduced in \cite{DIJ}. It gives a similar way to partition $\mathbb{R}^l$ from the perspective of silting theory.
Classical silting theory is mainly concerned with silting objects in ``small'' triangulated categories, such as the category of perfect complexes $\K^b(\proj A)$ over a finite dimensional algebra $A$ or other Hom-finite, Krull-Schmidt triangulated categories. Classical silting objects can be considered as generalisations of tilting complexes, which govern equivalences between derived categories of rings. They have a rich theory of mutations connected to cluster algebras and they parametrise certain well-behaved t-structures and co-t-structures. For instance, the Koenig-Yang correspondence states that equivalence classes of silting objects in $\K^b(\proj A)$ are in bijection with bounded t-structures with length hearts in the bounded derived category $\D^b(\fmod A)$ \cite{KY}. This provides a tight connection between such silting objects and the wall and chamber structure of the Bridgeland stability condition manifold of $\D^b(\fmod A)$.

The $g$-vector fan  is a polyhedral fan in $\Ko_0(\proj A) \otimes_\mathbb{Z} \mathbb{R} \simeq \mathbb{R}^l$ spanned by the classes or $g$-vectors of two-term partial silting complexes (summands of two-term silting complexes) in $\Ko_0(\proj A)\simeq \Ko_0(\K^b(\proj A))$. 
The connection between the $g$-vector fan of an algebra and the wall and chamber structure in given in \cite{Asai, Bridgeland, BST}. The chambers correspond bijectivelty to two-term silting complexes or cones of maximal dimension in the $g$-vector fan and the walls between two neighboring chambers corresponds to irreducible mutations of two-term silting complexes or cones of codimension 1. 

However, not all walls in the wall and chamber structure can be interpreted via classical silting theory. For example, consider the path algebra of the Kronecker quiver $A=kQ$, $Q=1\rightrightarrows 2$. The wall and chamber structure of $A$ is well-know and is schematically depicted on Figure \ref{fig:my_labelI}. In this case there is one ray, which does not appear in the g-vector fan of $A$, but gives a wall in the wall and chamber structure, this ray corresponds to $\theta$-semistable modules with $\theta=a(1,-1)^t$, $a\geq 0$.  One of the motivations for this paper is to interpret such limiting walls from the perspective of silting theory. However, in order to achieve this goal we will have to leave the realm of classical silting theory and consider large silting objects.

\begin{figure}[!ht]
\centering
\begin{tikzpicture}[scale=0.80]
\tikzstyle{every node}=[font=\fontsize{10.9pt}{15.5pt}\selectfont]
\draw [-{Stealth[scale=1.5]}, ] (6.25,9.75) -- (13.75,9.75);
\draw [-{Stealth[scale=1.5]}, ] (10,6) -- (10,13.5);
\draw [dashed] (10,9.75) -- (13.75,6);
\draw  (10,9.75) -- (13.75,7.875);
\draw  (10,9.75) -- (11.875,6);
\node [font=\fontsize{10.9pt}{15.5pt}\selectfont, inner xsep=0.080cm, inner ysep=0.085cm, rounded corners=0.020cm] at (10.5,13.375) {$[P_1]$};
\node [font=\fontsize{10.9pt}{15.5pt}\selectfont, inner xsep=0.080cm, inner ysep=0.085cm, rounded corners=0.020cm] at (14.125,9.625) {$[P_2]$};
\draw  (10,9.75) -- (13.75,7.25);
\draw  (10,9.75) -- (12.5,6);
\draw  (10,9.75) -- (13.75,7);
\draw  (10,9.75) -- (12.75,6);
\node [font=\fontsize{10.9pt}{15.5pt}\selectfont, inner xsep=0.080cm, inner ysep=0.085cm, rounded corners=0.020cm] at (13.75,6.625) {\small $\vdots$};
\node [font=\fontsize{11.9pt}{15.5pt}\selectfont, inner xsep=0.080cm, inner ysep=0.085cm, rounded corners=0.020cm] at (13.25,6) {\small $\hdots$};


\node [font=\fontsize{11.9pt}{15.5pt}\selectfont, inner xsep=0.080cm, inner ysep=0.085cm, rounded corners=0.020cm] at (13.95,5.8) { $\theta$};
\end{tikzpicture}
\caption{The wall and chamber structure for the Kronecker quiver.}
\label{fig:my_labelI}
\end{figure}
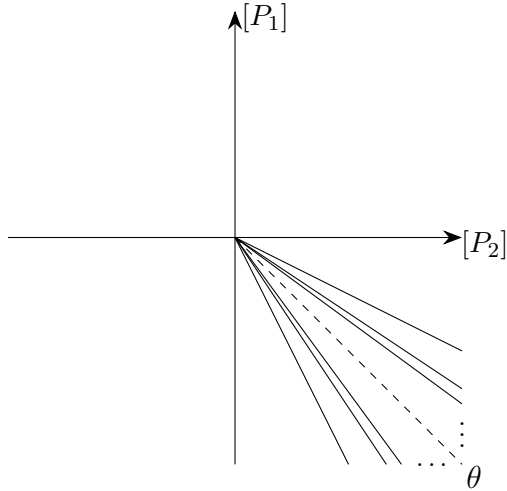

Large silting theory abstracts the notion of silting to more general triangulated categories, namely, triangulated categories with coproducts, such as the derived category of a ring $\D(A)$. 
Large silting objects can be also used to parametrise well behaved t-structures. In a compactly generated triangulated category, equivalence classes of silting objects are in bijection with nondegenerate cosmashing t-structures whose heart admits a projective generator \cite{NSZ}. The dual notion of cosilting can be used to parametrise nondegenerate smashing t-structures with the heart admitting an injective cogenerator. The theory of mutations of cosilting objects was recently extended to this general context in \cite{AHLSV}.

After their introduction in \cite{NSZ,PV}, a plethora of theoretical results concerning the behavior and properties of large silting and cosilting objects was obtained. However, there are no general systematic methods of constructing examples of large silting and cosilting objects even in the situation of the derived category of a ring. One of the main goals of this paper is to bridge this gap by considering colimits of silting and limits of cosilting objects.

Our constructions are inspired by the work of Buan and Solberg \cite{BS}, Braga and Coelho \cite{BC}, and Braga \cite{BC2}. 
There, the authors study when inverse limits of finitely generated cotilting modules of injective dimension at most $n$, and direct limits of tilting modules of projective dimension at most $n$, remain cotilting and tilting, respectively.
Unlike in the category of modules we cannot directly compute limits and colimits in the derived category. Instead we use homotopy limits and colimits defined via model structures on the category of chain complexes $\Ch(A)$.

We will consider four settings depending on the level of generality: countable directed colimits of silting objects over arbitrary rings; uncountable directed colimits of silting objects over right perfect or right noetherian rings; countable inverse limits of cosilting objects over arbitrary rings; and uncountable inverse limits of cosilting objects of cofinite type over left perfect rings. Note that right perfect rings include finite dimensional algebras over fields, so the results in this case will be applicable to our motivating problem.
The reason we have to impose additional conditions on the ring in the uncountable case is that the category of bifibrant objects for various model structures whose homotopy category gives the derived category of a ring is usually not closed under limits and colimits. 

We will further restrict our attention to $n$-silting and $n$-cosilting objects, which is in some sense analogous to the projective or injective dimension being at most $n$ in the case of tilting and cotilting modules. Such silting objects correspond to $n$-intermediate t-structures, that is t-structures whose aisles are sandwiched between the aisle of the standard t-structure and its shift by $n$. This is a technical assumption frequently used in the literature and guaranteeing additional control over the behavior of the t-structure. For example in the terminology of Neeman any t-structure in $\D(A)$ in the preferred equivalence class will be of this type up to shift \cite{Neem}. If $A$ is a finite-dimensional algebra, any lift of a bounded t-structure is $n$-intermediate up to shift \cite{MZ}. Similar to the silting case, an $n$-cosilting object gives an associated t-structure that is again $n$-intermediate up to shift. 
The utility of this condition can be, for example, demonstrated by \cite{MVSiltCosilt}, where the authors prove that any $n$-cosilting object is pure injective, which still remains an open question for general cosilting objects.

Let us describe the main results of the paper. For simplicity we will state the results only for the countable case, the case of an arbitrary ordinal can be found in Sections \ref{SecContSilt} and \ref{SecCoCont}. For an arbitrary ring $A$, we will consider a sequence $T_0,T_1,\dots$ of $n$-silting objects in $\D(A)$ such that the associated aisles form a descending chain.
The intersection $\U$ of the family of nested aisles $\U_i$ is an aisle of a t-structure corresponding to an $n$-silting object $T$. This can be seen by combining \cite[Proposition 6.3]{LV} and \cite[Theorem 4.6]{AMVSiltMod}. However it is not clear how to construct $T$ explicitly. This is the content of our first result.

\begin{theorem*}[\ref{ThmSilt}]
Let $A$ be a ring and let  $T_i$, $i\geq 0$ be a sequence of $n$-silting objects in $\D(A)$ with corresponding aisles $\U_i$. Assume that there is an inclusion of aisles $\U_{i+1} \subseteq \U_i$ for all $i\geq 0$.
Then there is a direct system 
\[
    T_0' \rightarrow T_1' \rightarrow T_2' \rightarrow \dots
\in \Ch(A)\] with $\Add(T_i)=\Add(T'_i)$ and such that $T := \hclim_{\mathbb{N}} T'_i$ is an $n$-silting object in $\D(A)$ with corresponding aisle $\U=\cap_{i\geq 0}\U_i$.
\end{theorem*}

The direct system in the category of chain complexes $\Ch(A)$ is constructed explicitly in Corollary \ref{Cor: diagram Frob}. For that we use the fact that the co-t-structures corresponding to the silting objects $T_i$ can be lifted to give cotorsion pairs on the category of bifibrant objects by \cite{SS}. Note that the condition $\Add(T_i)=\Add(T'_i)$ guarantees that $T_i$ and $T'_i$ generate the same t-structure, which is exactly the equivalence relation usually considered for silting objects, so replacing $T_i$ by $T'_i$ does not lead to any loss of generality.

The results in the case of a right perfect or a right noetherian ring are quite similar in nature. They describe an $n$-silting complex corresponding to a family of nested aisles parametrised by an arbitrary ordinal $\mu$ as a homotopy colimit of a continuous direct $\mu$-system of the corresponding $n$-silting objects, see Theorem \ref{contcolim}.

Let us now turn to $n$-cosilting objects.
In the countable case we get the following result, the corresponding diagram is constructed in Corollary \ref{Cor: co diagram Frob}.

\begin{theorem*}[\ref{ThmCoSilt}] 
Let $A$ be a ring and let  $C_i$, $i\geq 0$ be a sequence of $n$-cosilting objects in $\D(A)$ with the corresponding coaisles $\V_i$. Assume that there is an inclusion $\V_{i+1} \subseteq \V_i$ for all $i\geq 0$.
There exists an inverse system 
\[
    \dots \rightarrow C_2' \rightarrow C_1'\rightarrow C_0' \in \Ch(A)
\]
 with $\Prod(C'_i)=\Prod(C_i)$ such that $C := \lim_{\mathbb{N}^{op}}  C_i$ is  an $n$-cosilting object in $\D(A)$ with  $C\simeq \hlim_{\mathbb{N}^{op}}  C_i$ and with the corresponding coaisle $\V=\cap_{i\geq 0}\V_i$.
\end{theorem*}

Note that the condition $\Prod(C'_i)=\Prod(C_i)$ defines the equivalence of cosilting objects and implies that the t-structures corresponding to $C_i$ and $C'_i$ coincide. Unlike the silting case, where the exactness of directed colimits guarantees that the colimit in the category of chain complexes gives the homotopy colimit in $\D(A)$, in the cosilting case we need to check that the diagram constructed in Corollary \ref{Cor: co diagram Frob} is fibrant to guarantee that the limit in $\Ch(A)$ describes the homotopy limit in $\D(A)$.

In the uncountable case, the condition for the subcategory of bifibrant objects to be closed under the desired inverse limits seems to be quite restrictive. To circumvent this problem we consider $n$-cosilting objects of cofinite type, these cosilting objects correspond to compactly generated t-structures and are in bijection with silting objects over the opposite ring \cite{AHpara, BirdWill}. 
Our results in this case hold for left perfect rings. They describe an $n$-cosilting complex corresponding to a family of nested coaisles parametrised by an arbitrary ordinal $\mu$ as a homotopy colimit of a continuous inverse $\mu$-system of the corresponding $n$-cosilting objects of cofinite type, see Theorem \ref{ThmCoCont}.

In Section \ref{SecApp} we apply the obtained results to the study of the wall and chamber structure of the real Grothendieck group of a finite-dimensional algebra, thus connecting back to the world of classical silting theory. For a finite-dimensional algebra $A$ any vector $\theta$ in $\Ko_0(\proj A) \otimes_\mathbb{Z} \mathbb{R}\simeq \mathbb{R}^l$ gives rise to a numerical torsion pair $(\overline{\T}_\theta,\F_\theta)$ and to the corresponding t-structure via the Happel-Reiten-Smal{\o} tilt. The torsion pair $(\overline{\T}_\theta,\F_\theta)$ is tightly connected to the subcategory of $\theta$-semistable modules over $A$. If $\theta$ is the class of a two-term silting complex $T\in \K^b(\proj A)$ in the Grothendieck group $\Ko_0(\K^b(\proj A))\simeq \Ko_0(\proj A)$, the corresponding t-structure is the t-structure associated to $T$. Thus, moving around  $\mathbb{R}^l$ gives a way of producing sequences of two-term silting complexes $T_i$ in $\K^b(\proj A)$ whose aisles are nested whenever their $g$-vectors satisfy $\theta^{i+1}\leq \theta^i$. As scaling the vector $\theta$ by a positive multiple does not change the torsion pair $(\overline{\T}_\theta,\F_\theta)$, we can replace the vectors $\theta$ with $\epsilon\theta$, for $\epsilon>0$ to allow more freedom. 
The main result of this section states that  the homotopy colimit of two-term classical silting objects can be used to describe the numerical torsion pair corresponding to the limiting ray in the wall and chamber structure, which answers our initial motivating question. 

\begin{theorem*}[\ref{LimSilt}]
Let $A$ be a finite dimensional algebra over an algebraically closed field.    Let $\theta \in \Ko_0(\proj A)\otimes_\mathbb{Z}\mathbb{R}$.
    Let $T_i \in \K^b(\proj A)$, $i\geq 0$ be two-term silting complexes such that $[T_i]=\epsilon_i\theta^i$  for some  $\epsilon_i\in \mathbb{R}_{>0}$  and  \[\theta \leq \dots\leq \theta^{i+1}\leq \theta^i\leq \dots \leq \theta^0 \text{ with } \underset{i\rightarrow \infty}{\mathrm{lim}} \theta^i=\theta.\] 
    Let $T\in \D(A)$ be the two-term silting complex constructed from the sequence $T_i$ according to the algorithm described in Theorem \ref{ThmSilt}. Then $\U_{T}\cap \fmod A =\overline{\T}_{\theta}$.
\end{theorem*}

This has two further applications. If in the situation above $\theta$ is the $g$-vector of a two-term  partial silting complex $U$ in $\K^b(\proj A)$, then $T$ constructed from the sequence $T_i$ via Theorem \ref{ThmSilt} is additively equivalent to the Bongartz completion of $U$ (see Corollary \ref{bongartz}). Finally, if in addition we assume that $A$ is tame, then, combining our results with \cite{PlamondonYurikusa}, we can construct a two-term silting complex $T$ corresponding to any numerical torsion pair $(\overline{\T}_\theta,\F_\theta)$  via the algorithm described in Theorem \ref{ThmSilt}, see Corollary \ref{CorTame}.  

\subsection*{Acknowledgments}
The authors are grateful to Isaac Bird, Mikhail Gorsky, and Georgios Dalezios for many valuable conversations.
RL was partially supported by the Project 2022S97PMY Structures for Quivers, Algebras and Representations (SQUARE) funded by NextGenerationEU under NRRP, Call PRIN 2022 No. 104 of February 2, 2022 of Italian Ministry of University and Research, and by the project LAVIE - Large views of small phenomena: decompositions, localizations, and representation type, FIS 00001706, funded by Program FIS2021 of Italian Ministry of University and Research. RL is a member of the network INdAM-G.N.S.A.G.A. AZ acknowledges the support of the Institute of Mathematics, Czech Academy of Sciences (RVO 67985840), Q100192601 Lumina
quaeruntu, and GA\v{C}R project 26-22734S.

\section{Preliminaries}

\subsection{Notation}

Throughout this paper $A$ will denote a unital associative ring and we will denote by $\Mod A$ the category of right $A$-modules. By $\fmod A$, $\Proj A$, $\Inj A$, $\proj A$, and $\inj A$ we will denote the subcategories of finitely presented, projective, injective, finitely generated projective and finitely generated injective modules, respectively. We will denote the opposite ring by $A^{op}$.

The category of chain complexes of $A$-modules will be denoted by $\Ch (A)$, note that we will use the cohomological notation for chain complexes. The unbounded derived category of all $A$-modules will be denoted by $\D(A)=\D(\Mod A)$. We will use the notation $\K^b(\Proj A)$ and $\K^b(\Inj A)$ for the homotopy categories of bounded complexes of projective and injective modules, $\K^b(\proj A)$ and $\K^b(\inj A)$ will denote the homotopy categories of bounded complexes of finitely generated projective and finitely generated injective modules. 

For an additive category $\A$ and a class of objects $\X \subseteq \A$ we will denote by $\X^{\perp}$ the subcategory with objects $\{Y\in \A \mid \Hom_\A(\X,Y)=0\}$. The subcategory ${}^{\perp}\X$ is defined dually.
If $\A$ has coproducts, $\Add \X$ will denote the subcategory consisting of direct summands of coproducts of objects in $\X$. If $\A$ has products, $\Prod \X$ will denote the subcategory consisting of direct summands of products of objects in $\X$. The subcategory consisting of direct summands of finite coproducts of objects in $\X$ will be denoted by $\add \X$. Note that when we use the term subcategory we mean a subcategory closed under isomorphism. 

Let $\Tt$ be a triangulated category. We will usually denote the shift functor on $\Tt$ by $[1]$. For full subcategories $\U,\V \subseteq \Tt$ we will use the following notation
\[
\U\star \V=\{X\in \Tt \mid \text{ there exists a triangle } U\rightarrow X \rightarrow V \rightarrow \text{ with } U\in \U, V\in \V\}.
\]
For a class of objects $\X \in\Tt$ we will denote by $\tria(\X)$ the smallest triangulated subcategory of $\Tt$ containing $\X$, and by $\thick(\X)$ the smallest thick subcategory containing $\X$, that is the smallest triangulated subcategory of $\Tt$ congaing $\X$ and closed under direct summands. We say that $\X$ \emph{generates} $\Tt$ if $\Hom_{\Tt}(\X,Y[n])=0$ for all $n\in \Zn$ implies $Y=0$. Dually, we say that $\X$ \emph{cogenerates} $\Tt$ if $\Hom_{\Tt}(Y,\X[n])=0$ for all $n\in \Zn$ implies $Y=0$.

\subsection{Hom-orthogonal pairs, t-structures and co-t-structures}

A key tool in the subsequent sections of the paper will be that of a Hom-orthogonal pair of subcategories. In particular, we will make use of t-structures, introduced by Beilinson, Bernstein, Deligne and Gabber \cite{BBDG}, and co-t-structures, introduced by Bondarko \cite{Bon1}, under the name of weight structures, and Pauksztello \cite{Pauksztello}. Let us recall some basic definitions mainly in order to fix the notation. 

\begin{definition}
A pair of full subcategories $(\U,\V)$ of a triangulated category $\Tt$ is 

\begin{itemize}
    \item a \emph{Hom-orthogonal pair} if $\U^{\perp}=\V$ and $\U={}^{\perp}\V$;
    \item a \emph{t-structure} if it is a Hom-orthogonal pair, $\U[1] \subseteq \U$, and $\U\star \V=\Tt$;
    \item a \emph{co-t-structure} if it is a Hom-orthogonal pair, $ \V[1] \subseteq \V$ and $\U\star \V=\Tt$.
\end{itemize}    
\end{definition}

The left hand side $\U$ of a t-structure is called an \emph{aisle} and the right hand side $\V$ is called a \emph{co-aisle}.
The \emph{heart} $\mathcal{H}$ of a t-structure $(\U,\V)$ is the abelian category $\U \cap \V[1]$. The \emph{coheart} of a co-t-structure is the additive category $\mathcal{C}=\U[1]\cap \V$. 

If $(\U,\V)$ is a t-structure, the condition $\U\star \V=\Tt$ implies that for any object $X\in\Tt$ there is a triangle
\[
U\rightarrow X \rightarrow V \rightarrow \text{ with } U\in \U, V\in \V.
\]
This triangle is functorial and the maps $U\rightarrow X$ and $X \rightarrow V$ are a $\U$-cover and a $\V$-envelope for each $X\in\Tt$. The corresponding triangle guaranteed by the definition of a co-t-structure is not functorial and the corresponding maps are only a $\U$-precover and a $\V$-preenvelope.
In the context of this paper it will frequently happen that we consider a pair of a t-structure and a co-t-structure, which share one of the classes. For example, we can consider a triple $(\W,\U,\V)$, where $(\W,\U)$ is a co-t-structure and $(\U,\V)$ is a t-structure. In that case we say that the co-t-structure $(\W,\U)$ is \emph{left adjacent} to the t-structure $(\U,\V)$. The definition of a \emph{right adjacent} co-t-structure is analogous.

\subsection{Silting and cosilting objects}\label{SiltIntro}

We will introduce silting and cosilting objects only in the context of derived categories. For more general notions see \cite{NSZ, PV}. For an object $S\in \D(A)$ we can consider the following subcategories of $\D(A)$:  
\[
\U_S=\{X\in \D(A)\mid \Hom_{\D(A)}(S,X[>0])=0\},
\]
\[
\V_S=\{X\in \D(A)\mid \Hom_{\D(A)}(S,X[\leq 0])=0\}.
\]

\begin{definition}
An object $S$ is called \emph{silting} if the pair $(\U_S,\V_S)$ is a t-structure.      
\end{definition}

 Note that such an object $S$ is necessarily a generator of $\D(A)$. 
Since the triangulated category we are working in is compactly generated, a left adjacent co-t-structure exists for any silting object (see \cite[Corollary 3.10]{AMVsmashing}, \cite[Theorem 3.2.4]{Bondarko}). We will denote this co-t-structure $(\W_S, \U_S)=({}^{\perp}\U_S,\U_S)$. The coheart of this co-t-structure coincides with $\Add(S)$. We will call two silting objects $S,S'$ \emph{equivalent} if they generate the same t-structure, which happens if and only if $\Add{S}=\Add{S'}$. The subcategory $\U_S$ is the smallest subcategory of $\D(A)$ containing $S$ and closed under positive shifts, coproducts and extensions \cite{AJSS}.

We will work with \emph{$n$-silting objects}, that is silting objects in $\D(A)$ isomorphic to a complex of projective $A$-modules concentrated in degrees $[-n+1, 0]$. By \cite[Proposition 4.2]{AMVSiltMod}, such a complex $S$ is silting exactly when the following conditions hold:
\begin{enumerate}
    \item $\Hom_{\D(A)}(S,S^{(J)}[>0])=0$ for any set $J$,
    \item $\tria(\Add(S))=\K^b(\Proj A)$.
\end{enumerate}

A canonical example of a silting object is the stalk complex $A$ concentrated in degree $0$. In this case one gets the \emph{standard t-structure} $(\U_A, \V_A)=(\D^{\leq 0},\D^{> 0})$. Here $\D^{\leq 0}$ denotes the subcategory of $\D(A)$ consisting of objects with zero cohomology in positive degrees and $\D^{> 0}$ denotes the subcategory of $\D(A)$ consisting of objects with zero cohomology in degrees $\leq 0$. For the shifts of the aisle and the coaisle of the standard t-structure we will often use the notation $\D^{\leq -n}=\D^{\leq 0}[n]$ and $\D^{> -n}=\D^{> 0}[n]$. 

From degree considerations for an $n$-silting object $S$ we get the following inclusions: 
\[
\D^{\leq -n+1}\subseteq \U_S \subseteq \D^{\leq 0}.
\]

A t-structure $(\U,\V)$ for which the inclusions $\D^{\leq -n+1}\subseteq \U \subseteq \D^{\leq 0}$ hold will be called \emph{$n$-intermediate.}
Let $S$ be an $n$-silting object. We can consider a sequence of triangles that gradually build $A$ from $\Add(S)$:
\[
S_{i+1}'[-1] \rightarrow S_i' \rightarrow S_{i+1} \rightarrow S_{i+1}',
\]
where the triangles are the approximation triangles associated to the co-t-structure  $(\W_S, \U_S)$ and
\begin{enumerate}
    \item $S_0'=A$,
    \item $S_i\in \Add(S)$,
    \item $S_i' \in \W_S[1]\cap \U_S[1-n+i]$,
    \item $S_n=S_{n-1}'$, $S_{n}'=0$.
\end{enumerate}

From this, taking the compositions $S_i\rightarrow S_{i}'\rightarrow  S_{i+1}$ one can construct a sequence of morphisms 
\begin{equation}\label{SiltSeq}
 A \xrightarrow{g_1} S_1 \xrightarrow{g_2} S_2 \xrightarrow{g_3} \dots \xrightarrow{g_n} S_n   
\end{equation}

with each $S_i\in \Add(S)$.
Moreover, we have that $T:= \oplus_{i=1}^n S_i$ is an $n$-silting object equivalent to $S$. Indeed, this follows by \cite[Theorem 1 (2)]{NSZ}, since $T$ is a generator.

Dually, for an object $C\in \D(A)$ we can consider the following two subcategories: 
\[
\U^C=\{X\in\D(A)\mid \Hom_{\D(A)}(X,C[\leq 0])=0\},
\]
\[
\V^C=\{X\in\D(A)\mid \Hom_{\D(A)}(X,C[> 0])=0\}.
\]

\begin{definition}
An object $C$ in a is called \emph{cosilting} if the pair $(\U^C,\V^C)$ is a t-structure.      
\end{definition}

In this case the object $C$ is necessarily a cogenerator. The subcategory $\V^C$ is the smallest co-aisle containing $C$ \cite[Proposition 4.9]{PV}. Two cosilting objects $C$ and $C'$ are \emph{equivalent} when $(\V^C,\W^C)=(\V^{C'},\W^{C'})$ or equivalently when $\Prod(C)=\Prod(C')$.

We will work with \emph{$n$-cosilting objects}, that is cosilting objects isomorphic to a complex of injective $A$-modules concentrated in degrees $[0, n-1]$. Such a complex $C$ is cosilting if and only if the following two conditions hold
\begin{enumerate}
    \item $\Hom_{\D(A)}(C^J,C[>0])=0$ for any set $J$,
    \item $\tria(\Prod(C))=\K^b(\Inj A)$. 
\end{enumerate}

We include a short proof of this fact for the convenience of the reader. The if implication can be found in \cite[Proposition 3.10]{MVSiltCosilt}, as $\tria(\Prod(C))=\K^b(\Inj A)$ implies $\thick(\Prod(C))=\K^b(\Inj A)$. The only if implication follows from the fact that $C\in \V^C$ and $\V^C$ is closed under products as a coaisle. The second condition follows from the fact that the t-structure $(\U^C,\V^C)$ is intermediate and the shift of the standard t-structure $(\D^{\leq -1},\D^{> -1})$ corresponds to the injective cogenerator $E$ of $\Mod A$ considered as a cosilting object. So, analogously to the silting case, one can construct $E$ and any injective module inductively from $\Prod(C)$ as we describe below.

Let $(\U^C,\V^C)$ be a t-structure associated to an $n$-cosilting object $C$ in $\D(A)$.
The right adjacent co-t-structure to $(\U^C,\V^C)$ exists and we will denote it by $(\V^C,\W^C)=(\V^C,(\V^C)^{\perp})$ (see \cite[Corollary 3.2.6]{Bondarko}, \cite[Theorem 3.13]{MVSiltCosilt}). The negative shift of the coheart of this co-t-structure coincides with $\Prod(C)$, i.e., $(\V^C[1]) \cap \W^C = \Prod(C)[1]$. 

For an $n$-cosilting object $C$, we get 
\[
\D^{\geq n-1}\subseteq \V^C \subseteq \D^{\geq 0}.
\]

As in the silting case let us construct $E$ from $\Prod(C)$. For each $0 \leq i \leq n-1$, consider the approximation triangle of $C_i'$ corresponding to the co-t-structure $(\V^C,\W^C)$:
\[ C_{i+1} \rightarrow C_i' \rightarrow C_{i+1}'[1] \rightarrow C_{i+1}[1],
\] 
with
\begin{enumerate}
\item $C_0' := E$,
\item $C_i \in \Prod(C)$,
\item $C_{i+1}' \in \V^C[n-i-2] \cap \W^C[-1]$,
\item $C_n = C_{n-1}'$, $C_{n}'=0$.
\end{enumerate}
We have a sequence of morphisms:
\[
C_n \overset{f_n}{\longrightarrow} C_{n-1} \overset{f_{n-1}}{\longrightarrow} \dots \overset{f_3}{\longrightarrow} C_2 \overset{f_2}{\longrightarrow}C_1 \overset{f_1}{\longrightarrow} E
\] with $C_i \in \Prod(C)$ where the maps are obtained as compositions $C_{i+1}\rightarrow C'_{i}\rightarrow C_{i}$. The object  $Q:=\oplus_1^n C_{i}$ is a cosilting object by \cite[Proposition 3.10]{MVSiltCosilt} since it is an $n$-term complex of injectives. Since $Q\in \Prod(C)$, we get that $Q$ and $C$ are equivalent cosilting objects.

\subsection{Cotorsion pairs in exact categories}
By \cite{SS}, t-structures and co-t-structures in $\D(A)$ can be described via cotorsion pairs in an exact Frobenius category whose stable category is equivalent to $\D(A)$. Let us recall this correspondence and some definitions related to cotorsion pairs which will be used later.
We refer the reader to \cite{buhler} for details on exact categories. 

Let $\Ee$ be an exact category.  We will denote admissible short exact sequences in $\Ee$ as short exact sequences, $\Ext^1_{\Ee}(X,Y)$ will denote the corresponding Ext-groups.
Let $\X$ be a class of objects in $\Ee$, we will use the standard notation for its Ext-orthogonal subcategories:
\[\X^{\perp_1}=\{Y\in \Ee \mid \Ext^1_{\Ee}(\X,Y)=0 \},\] \[{}^{\perp_1}\X=\{Y\in \Ee \mid \Ext^1_{\Ee}(Y,\X)=0 \}.\]  

We call a pair $(\X, \Y)$ of full subcategories of $\Ee$ a \emph{cotorsion pair} if $\X^{\perp_1}=\Y$ and $\X={}^{\perp_1}\Y$. A cotorsion
pair $(\X, \Y)$ is called \emph{complete} if for any $M \in \Ee$ we have two admissible short exact sequences
\[
0\rightarrow Y\rightarrow X \rightarrow M \rightarrow 0
\]
\[
0\rightarrow M\rightarrow Y' \rightarrow X' \rightarrow 0
\]
where $X,X'\in \X$ and $Y,Y'\in \Y$. The admissible epi $X \rightarrow M$ is
 an $\X$-precover and the admissible mono $M \rightarrow  Y'$
is a $\Y$-preenvelope. The two sequences whose existence is guaranteed by the definition of the complete cotorsion pair are called \emph{approximation sequences}. 

For a class of objects $\X$ an $\X$-precover $X\rightarrow M$ of $M$ is called \emph{special} if it fits into an admissible short exact sequence $0\rightarrow Y\rightarrow X \rightarrow M \rightarrow 0$ with $Y\in \X^{\perp_1}$. Dually for a class of objects $\Y$ a $\Y$-preenvelope $M\rightarrow Y$ of $M$ is called \emph{special} if it fits into an admissible short exact sequence $0\rightarrow M\rightarrow Y \rightarrow X \rightarrow 0$ with $X\in {}^{\perp_1}\Y$. As we see complete cotorsion pairs guarantee the existence of special $\X$-precovers and $\Y$-preenvelopes. 

A class  $\X\subseteq \E$ is called \emph{projectively
resolving} if $\X$ contains all projective objects of $\E$, $\X$ is closed under
extensions, and $\X$ is closed under taking kernels of admissible epis. \emph{Injectively coresolving} classes are defined dually. A cotorsion pair $(\X,\Y)$ is called \emph{hereditary} if $\X$ is projectively resolving and $\Y$ is injectively coresolving.

In case $\Ee$ is a Frobenius exact category we will denote by $\Omega$ the syzygy in $\Ee$ which descends to the inverse of the shift functor on the stable category $\underline{\Ee}$. 
In this case a cotorsion pair $(\X,\Y)$ is hereditary if and only if $\Omega \X \subseteq \X$. For a subcategory $\X$ of $\Ee$ we will denote $\underline{\X}$ the essential image of $\X$ under the additive quotient functor from $\Ee$ to $\underline{\Ee}$. We will use the aforementioned correspondence from \cite{SS} only for the case of co-t-structures.

\begin{proposition}\cite[Proposition 3.16]{SS}
Let $\Ee$ be a Frobenius exact category. The assignment $(\F, \T ) \mapsto (\underline{\F}[-1] ,\underline{\T})$ gives a bijective correspondence between complete hereditary cotorsion pairs in $\Ee$ and co-t-structures in $\underline{\Ee}$.   
\end{proposition}

Starting from a co-t-structure in $\underline{\Ee}$, the classes of the corresponding cotorsion pair in $\Ee$ are given by the preimages of the corresponding subcategories of the co-t-structure under the quotient functor.  Note that under the assignment above the cotorsion approximation sequence 
\[0\rightarrow X\rightarrow T \rightarrow F \rightarrow 0\]
which exists for any $X\in \Ee$ becomes the approximation triangle of $X$
\[F[-1]\rightarrow X\rightarrow T \rightarrow F  \]
with respect to the co-t-structure $(\underline{\F}[-1] ,\underline{\T})$.

\subsection{Model structures on chain complexes}
Throughout this paper we will be concerned with homotopy limits and colimits in $\D(A)$. To describe them we will use the language of model categories.
For a ring $A$, there are multiple examples of abelian model structures on the category $\Ch (A)$ of chain complexes over $A$ whose homotopy category is the derived category $\D(A)$. In particular, the projective model structure on $\Ch (A)$  and the injective model structure on $\Ch (A)$ are of that type. For details on model categories see \cite{hovey, gillespie}. We summarise the main properties that we will use below.

An \emph{abelian model structures} on a complete and cocomplete abelian category can be given by a \emph{Hovey triple} $(Q,W,R)$. That is a triple of classes of objects such that $(Q\cap W,R)$ and $(Q,W\cap R)$ are complete cotorsion pairs and $W$ satisfies 2 out of 3 property on short exact sequences. Objects in $Q$ are called \emph{cofibrant}, objects in $R$ are called \emph{fibrant}, and objects in $W$ are called \emph{trivial}. The \emph{cofibrations} are defined to be monomorphisms with cokernel in $Q$ and the \emph{fibrations} are defined to be epimorphisms with kernel in $R$. The objects in $Q\cap R$ are \emph{bifibrant}. An abelian model structure is called \emph{hereditary} if both cotorsion pairs $(Q\cap W,R)$ and $(Q,W\cap R)$ are hereditary. For a hereditary abelian model structure the subcategory of bifibrant objects $Q\cap R$ is an exact Frobenius category whose stable category is equivalent to the homotopy category of the model category via the fibrant-cofibrant replacement.

In later sections we will be working with an abstract hereditary abelian model structure, but will always have the projective and the injective model structures in mind as our main examples. The projective model structure is given by the Hovey triple $(\dgP, \Ac , \All)$, where $\All$ denotes the class of all chain complexes of $A$ modules, $\Ac$ denotes the class of acyclic complexes, and $\dgP$ denotes the class of \emph{DG-projective complexes}, that is complexes $P$ such that each term is projective and any chain map from $P$ to an acyclic complex is null homotopic. In particular, any bounded above complex of projectives is DG-projective. All chain complexes are fibrant in this model structure and DG-projective complexes are cofibrant. 

Dually, the injective model structure is given by the Hovey triple $(\All, \Ac , \dgI)$, where $\dgI$ denotes the class of \emph{DG-injective complexes}, that is complexes $I$ such that each term is injective and any chain map to $I$ from an acyclic complex is null homotopic. In particular, any bounded below complex of injectives is DG-injective. For this model structure all chain complexes are cofibrant and DG-injective complexes are fibrant.

Both the projective and the injective model structures are hereditary. So, the derived category $\D(A)$ is equivalent to the stable category of the exact Frobenius category $\dgP$, in which projective-injective objects are the acyclic DG-projective complexes. Dually, $\D(A)$ is equivalent to the stable category of the exact Frobenius category $\dgI$, in which the projective-injective objects are the acyclic DG-injective complexes.

\subsection{Homotopy limits and colimits}\label{SecHomLimColim}

Let $I$ be a small category. We will usually consider the case when $I$ is the category $\mathbb{N}$, $\mathbb{N}^{op}$ or it is given by another ordinal.
Let 
\[ \Ch(A) \xrightarrow{\Delta_I}  \Ch( A)^{I}\]
be the constant diagram functor from the category of chain complexes to the category of $I$-shaped diagrams in $\Ch(A)$ (which is an abelian category). The left adjoint to this  functor is the colimit and the right adjoint is the limit.

The derived category $\D(\Mod A^{I})$ of the abelian category $\Mod A^{I}$ of $I$-shaped diagrams in $\Mod A$ can be constructed as the localisation of the category $\Ch( A)^{I}$ with respect to quasi-isomorphisms, which are defined objectwise. The functor $\Delta_I$ induces a functor on the level of derived categories that sends an object $X$ in  $\D(A)$ to the constant diagram of shape $I$ in $\D(\Mod A^{I})$, consisting of the copies of $X$ and identity morphisms. 

 The \emph{homotopy colimit} is the left derived functor of the colimit and the \emph{homotopy limit} is the right derived functor of the limit. We will use $\hclim_I$ to denote homotopy colimit, and $\hlim_I$ to denote the homotopy limit. Since directed colimits are exact, $\hclim_I$ can be computed as colimit in the category of chain complexes in case $I$ is directed, which is the only case we will consider. 

The functor $\hlim_I$ is a little bit more involved.
Since it is the right derived functor of the limit functor, in order to compute it we will need the description of fibrant objects in $\Ch(A)^I$ for some model structure giving $\D(\Mod A^{I})$ as its homotopy category. This will be applied only to diagrams of shape $I=\mu^{op}$, for an ordinal $\mu$, so it will suffice to use the description of fibrant objects for inverse categories $I$ in the sense of \cite[Definition 5.1.1]{hovey}. For a fibrant object in $\Ch(A)^I$ the functor $\hlim_I$ is then computed as limit in $\Ch(A)$.

\begin{theorem}\cite[Theorem 5.1.3]{hovey}\label{Thm: diag model structure} Given a model category $C$ and an inverse category $I$, we have a model
structure on $C^I$ where the weak equivalences and cofibrations are the objectwise
ones, and a map $f : X \rightarrow Y$ is a (trivial) fibration if and only if the induced map
$X_i \rightarrow Y_i \times_{M_iY} M_iX$ is a (trivial) fibration for all $i\in I$.
\end{theorem}

Here $M_i$, the \emph{matching space functor}, is defined as the composite
\[M_i : C^I \rightarrow C^{(I,i)} \xrightarrow{lim} C,\]
where $(I,i)$ is the category whose objects are all non-identity maps $i\rightarrow j$ in $I$ and whose morphisms are commutative triangles, and $C^I \rightarrow C^{(I,i)}$ is the restriction. We have a natural transformation $X_i \rightarrow M_iX$.

For any ordinal $\mu$, we consider the category $\mathbb{\mu}^{op}$ with objects given by ordinals $\alpha < \mu$ such that there is a unique arrow $\alpha \rightarrow \beta$ whenever $\alpha \geq \beta$. 
We get that 
\[
(\mathbb{\mu}^{op},\alpha)= \alpha^{op}
\]
and the functor $M_\alpha$ sends a diagram $X$ in $C^{\mathbb{\mu}^{op}}$ to $X_{\alpha-1}$ if $\alpha$ is a successor ordinal and to $\lim_{\beta< \alpha} X_\beta$ if $\alpha$ is a limit ordinal. 
The natural transformation $X_\alpha \rightarrow M_\alpha X$ is given by the arrow in the diagram $X_\alpha \rightarrow X_{\alpha-1}$ if $\alpha$ is a successor ordinal and the limit morphism $X_\alpha \to \lim_{\beta< \alpha} X_\beta$ if $\alpha$ is a limit ordinal.

The fibrant objects are determined by the fibrations; indeed they are the objects $X$ such that canonical map to the zero object is a fibration. We can therefore use Theorem \ref{Thm: diag model structure} to identify the fibrant objects. In our setting, the zero object $Z=0 \in C^{\mathbb{\mu}^{op}}$ is given by $Z_\alpha=0$ for all $\alpha< \mu$. For successor ordinals $\alpha<\mu$ up to isomorphism we get that the following map is a fibration
\[(X_\alpha \rightarrow Z_\alpha \times_{M_\alpha Z} M_\alpha X)=(X_\alpha \rightarrow 0 \times_{0} X_{\alpha-1})=(X_\alpha \rightarrow X_{\alpha-1}).\]
For a limit ordinal $\alpha <\mu$ the following map is a fibration
\[(X_\alpha \rightarrow Z_\alpha \times_{M_\alpha Z} M_\alpha X)=(X_\alpha \rightarrow 0 \times_{0} \lim_{\beta< \alpha} X_\beta)=(X_\alpha \rightarrow \lim_{\beta< \alpha} X_\beta)\]
where the last morphism is the limit morphism. 
Additionally, if one checks the condition for $i=0$, we get that $X_0$ is fibrant, since the map $X_0\rightarrow 0$ must be a fibration. So the fibrant objects in $C^{\mu^{op}}$ are the diagrams of shape $\mu^{op}$ such that 
\begin{enumerate}
 \item $X_0$ is fibrant; 
\item each map $X_\alpha \rightarrow X_{\alpha-1}$ for $\alpha$ successor ordinal is a fibration; and 
\item each map $X_\alpha \rightarrow \lim_{\beta< \alpha} X_\beta$ for $\alpha$ a limit ordinal is a fibration.
\end{enumerate} 
Later we will be interested in $\mu^{op}$-shaped diagrams $X \in C^{\mu^{op}}$ that are \emph{continuous inverse systems}, that is, for each limit ordinal $\alpha<\mu$, the limit $\lim_{\beta< \alpha} X_\beta$ exists 
and the limit morphism
$X_\alpha \rightarrow \lim_{\beta< \alpha} X_\beta$ is an isomorphism. For diagrams of this kind, the third condition guaranteeing that $X$ is fibrant is satisfied automatically.

 Since short exact sequences of chain complexes give triangles in the derived category of any abelian category and left and right total derived functors are exact we have that, if 
 \[0\rightarrow A\rightarrow B \rightarrow C\rightarrow 0\]
 is a short exact sequence in $\Ch(\Mod A^{I})\simeq \Ch(A)^{I}$, then 
 \[\hclim_I A \longrightarrow \hclim_I B \longrightarrow \hclim_I C \longrightarrow (\hclim_I A)[1]\]
 and  \[\hlim_I A \longrightarrow \hlim_I B \longrightarrow \hlim_I C \longrightarrow (\hlim_I A)[1]\]
 are triangles in $\D(A)$.

\section{Countable colimits of silting objects}\label{SecSilt}

In this section we will consider countable homotopy colimits of $n$-silting objects corresponding to a sequence of nested aisles in $\D(A)$ for a ring $A$.
We will work in the following setting: $(Q,W,R)$ is a hereditary abelian model structure on $\Ch(A)$ whose homotopy category is $\D(A)$. The category of bifibrant objects $Q\cap R =\E$ is exact Frobenius with the stable category $\underline{\E}$ equivalent to $\D(A)$. One particular case where computations are easier is the projective model structure with $\E=\dgP$ as $n$-silting objects are isomorphic to bounded complexes of projective modules. 

We will use the following notation: for an $n$-silting object $S$ in $\D(A)$, we denote the cotorsion pair in $\E$ associated to the co-t-structure $(\W_S, \U_S)$ via the bijection from \cite[Proposition 3.16]{SS} by $(\F_S, \T_S)$. In particular, $\W_S=\underline{\F}_S[-1]$ and $\U_S=\underline{\T}_S$. 
For a sequence of $n$-silting objects $T_0, T_1, T_2, \dots \in \D(A)$ we will denote the co-t-structures $(\W_{T_i}, \U_{T_i})$ by $(\W_i, \U_i)$ and the cotorsion pairs $(\F_{T_i}, \T_{T_i})$ by $(\F_i, \T_i)$. 

The aim of this section is to explicitly construct an $n$-silting object corresponding to the t-structure whose aisle is the intersection of a countable sequence of decreasing aisles corresponding to $n$-silting objects $T_0, T_1, T_2, \dots \in \D(A)$. So we will work in the following setup:

\begin{setup}\label{SiltSet}
    Let $T_0, T_1, T_2, \dots \in \D(A)$ be a sequence of $n$-silting objects such that  $T_{i+1}\in \U_i$ for $i\geq 0$. The last condition is equivalent to $\U_{i+1}\subseteq \U_i$, since $\U_{i+1}$ is the smallest aisle in $\D(A)$ containing $T_{i+1}$.
\end{setup}

For the rest of the section, by abuse of notation, $A$ denotes a fixed fibrant-cofibrant replacement of $A$ in the model structure $(Q, W, R)$.

\begin{proposition}\label{SiltApp}
Suppose $S$ and $T$ are $n$-silting objects in $\D(A)$ such that $T \in \U_S$.   
Then there exists a commutative diagram with exact rows in $\E$

\[\xymatrix{0 \ar[r] & A \ar[r] \ar@{=}[d] & S_1 \ar[r] \ar[d] & S_2 \ar[r] \ar[d] & \dots \ar[r] & S_n \ar[r] \ar[d] & 0  \\
0 \ar[r] & A \ar[r]  & T_1 \ar[r]  & T_2 \ar[r]  & \dots \ar[r] & T_n \ar[r] & 0}\] 

such that, for every $1\leq i\leq n$, we have $S_i\in \F_S\cap\T_S$, $T_i\in\F_T\cap \T_T$ and the vertical arrows are special $\T_T$-preenvelopes.    
\end{proposition}

\begin{proof}
In the category $\E$ the preimage of the Sequence \ref{SiltSeq} from Section \ref{SiltIntro} can be chosen to be spliced from admissible short exact sequences arising as the approximation sequences with respect to the cotorsion pair $(\F_S, \T_S)$. To introduce notation we will do it for both $S$ and $T$. In particular, 
consider the special approximation sequences of $A$ \[0 \to A \overset{\lambda}{\rightarrow} S_1 \overset{p_1}{\rightarrow} S_1' \rightarrow 0 \text{ and } 0 \to A \overset{\gamma}{\rightarrow} U_1 \overset{q_1}{\rightarrow} U_1' \rightarrow 0\]
induced by $(\F_S, \T_S)$ and $(\F_T, \T_T)$ respectively.  That is, $S_1\in \T_S$, $S_1'\in \F_S$, $U_1\in \T_T$ and $U_1'\in \F_T$.

Inductively define the special approximation sequences \[0 \to S_i' \overset{\lambda_i}{\rightarrow} S_{i+1} \overset{p_{i+1}}{\rightarrow} S_{i+1}' \rightarrow 0 \text{ and } 0 \to U_i' \overset{\gamma_i}{\rightarrow} U_{i+1} \overset{q_{i+1}}{\rightarrow} U_{i+1}' \rightarrow 0\] 
induced, again, by $(\F_S, \T_S)$ and $(\F_T, \T_T)$ respectively. So $S_{i+1}\in \T_S$, $S_{i+1}'\in \F_S$, $U_{i+1}\in \T_T$ and $U_{i+1}'\in \F_T$ for each $i\geq 1$.

In $\D(A)\simeq \underline{\E}$, the images of these admissible short exact sequences give us triangles as in the construction of Sequence \ref{SiltSeq}. Thus,  $S_{n-1}'\simeq S_n \in \D(A)$ and the image of $S_n'$ in $\D(A)$ is isomorphic to $0$. So $S_n'$ is a projective-injective object in $\E$ and the sequence $0\rightarrow S'_{n-1} \rightarrow S_n \rightarrow S_n'\rightarrow 0$ splits. We can chose $S_{n-1}'=S_n$ and $S_n'=0$ in $\E$. The same holds for the case of $T$ and we can choose $U_{n-1}'=U_n$ and $U_n'=0$ in $\E$. 
We also have $S_i \in \F_S\cap \T_S$, $U_i \in \F_T \cap \T_T$ for $1\leq i \leq n$, since the corresponding inclusions hold in $\D(A)$.

By assumption $T \in \U_S$, so we have $\U_T \subseteq \U_S$, and thus $\T_T \subseteq \T_S$.
Since $\lambda$ is a $\T_S$-preenvelope and $U_1\in \T_T \subseteq  \T_S$, there exists $s_1 \colon S_1 \to U_1$ such that $\gamma = s_1\lambda$.  This yields a commutative diagram 
\[\xymatrix{0 \ar[r] & A \ar[r]^\lambda \ar@{=}[d]  & S_1 \ar[r]^{p_1} \ar[d]^{s_1} & S_1' \ar[r] \ar[d]^{s_1'}  & 0  \\
0 \ar[r] & A \ar[r]^\gamma  & U_1 \ar[r]^{q_1}  & U_1' \ar[r]  & 0}\] 
and, inductively, we obtain commutative diagrams
\[\xymatrix{0 \ar[r] & S_{i}' \ar[r]^{\lambda_i} \ar[d]^{s_i'} & S_{i+1} \ar[r]^{p_{i+1}} \ar[d]^{s_{i+1}} & S_{i+1}' \ar[r] \ar[d]^{s_{i+1}'}  & 0  \\
0 \ar[r] & U_i' \ar[r]^{\gamma_i}  & U_{i+1} \ar[r]^{q_{i+1}}  & U_{i+1}' \ar[r]  & 0}\]
for $1\leq i <n$.

Next we modify these sequences so that the vertical maps are special $\T_T$-preenvelopes.

For each $1\leq i < n$, consider the special approximation sequence \[ 0 \to S_i' \overset{j_i}{\rightarrow} \overline{U_i} \rightarrow \overline{U_i'} \to 0,\] where $\overline{U_i} \in \T_T$ and $\overline{U_i'}\in \F_T$.  Since $S_i' \in \F_S \subseteq \F_T$, we have that $\overline{U_i}\in \T_T \cap \F_T$.  Using that $\lambda_i \colon S_i' \to S_{i+1}$ is a $\T_S$-preenvelope and that $\overline{U_i}\in \T_T \subseteq \T_S$, we also obtain $j_i' \colon S_{i+1} \to \overline{U_i}$ such that $j_i = j_i'\lambda_i$.  Using this data we can consider the following commutative diagrams with exact rows.
\[\xymatrix{0 \ar[r] & A \ar[rr]^\lambda \ar@{=}[d]  & & S_1 \ar[rr]^{p_1} \ar[d]^{t_1:=\left(\begin{smallmatrix} s_1\\j_1p_1 \end{smallmatrix}\right)} & & S_1' \ar[r] \ar[d]^{t_1':=\left(\begin{smallmatrix} s_1'\\j_1 \end{smallmatrix}\right)}  & 0  \\ 0 \ar[r] & A \ar[rr]_-{\left(\begin{smallmatrix} \gamma\\0 \end{smallmatrix}\right)} &  & U_1 \oplus \overline{U_1} \ar[rr]_-{\left(\begin{smallmatrix} q_1 & 0\\0 & 1 \end{smallmatrix}\right)}  & & U_1' \oplus \overline{U_1} \ar[r]  & 0,}\] 

\[\xymatrix{0 \ar[r] & S_{i}' \ar[rr]^{\lambda_i} \ar[d]^{t_i':=\left(\begin{smallmatrix} s_i'\\j_i \end{smallmatrix}\right)} & & S_{i+1} \ar[rr]^{p_{i+1}} \ar[d]^{t_{i+1}:= \left(\begin{smallmatrix} s_{i+1}\\j_i' \\ j_{i+1}p_{i+1} \end{smallmatrix}\right)} & & S_{i+1}' \ar[r] \ar[d]^{t_{i+1}' := \left(\begin{smallmatrix} s_{i+1}'\\j_{i+1} \end{smallmatrix}\right)}  & 0  \\
0 \ar[r] & U_i'\oplus \overline{U_i} \ar[rr]_-{\left(\begin{smallmatrix} \gamma_i & 0\\0 & 1 \\ 0 & 0 \end{smallmatrix}\right)}  & & U_{i+1}\oplus \overline{U_i} \oplus \overline{U_{i+1}} \ar[rr]_-{\left(\begin{smallmatrix} q_{i+1} & 0 & 0\\0 & 0 & 1 \end{smallmatrix}\right)}  & & U_{i+1}'\oplus \overline{U_{i+1}} \ar[r]  & 0.}\]

By virtue of the subcategories to which each object belongs to, we can deduce that the upper rows of the diagrams are special approximation sequences with respect to $(\F_S, \T_S)$ and the lower rows are special approximation sequences with respect to $(\F_T, \T_T)$. Setting $T_1 := U_1 \oplus \overline{U_1}$ and $T_{i+1} := U_{i+1}\oplus \overline{U_i} \oplus \overline{U_{i+1}}$ for $1\leq i < n$, we can construct the diagram in the statement of the proposition. Note that again we can choose $\overline{U_n}=0=\overline{U_n}'$.

Since we will splice these sequences together to obtain the diagram in the statement of the proposition, it remains to check that the middle vertical maps in these diagrams are special $\T_T$-preenvelopes.

The morphisms $t_i'$ are admissible monos, since so are $j_i$, and therefore also the morphisms $t_i$ are admissible monos.  Consider the admissible short exact sequences
\[  0 \to S_{i} \overset{t_{i}}{\longrightarrow} T_i \rightarrow Z_{i} \to 0 \] for $1\leq i \leq n$. We already know that $T_i \in \T_T$ and so it suffices to show that $Z_i\in \F_T$.  We will show that $\mathrm{Ext}^1_{\E}(Z_i, U)=0$ for all $U\in \T_T$ and hence that $Z_i \in \F_T$. Consider the long exact sequence 
\[\dots \to \mathrm{Hom}_{\E}(T_i, U) \overset{t_i^*}{\to} \mathrm{Hom}_{\E}(S_i, U) \to \mathrm{Ext}^1_{\E}(Z_i, U) \to \mathrm{Ext}^1_{\E}(T_i, U) \to  \dots .\]

Since $T_i\in \F_T$, we have that $\mathrm{Ext}^1_{\E}(T_i, U)=0$.  It is enough to show that $t_i^*$ is surjective.  We show that $t_1^*$ is surjective and explain how to modify the argument for $i>1$.  

Let $\alpha\in \mathrm{Hom}_{\E}(S_1, U)$ and consider $\alpha\lambda \colon A \to U$.  Using that $\gamma$ is a $\T_T$-preenvelope, we have that there exists $\alpha' \colon U_1 \to U$ such that $\alpha\lambda = \alpha'\gamma = \alpha's_1 \lambda$.  Then $0 = (\alpha -\alpha's_1 )\lambda$ so there exists a unique $\beta \colon S_1' \to U$ such that $\beta p_1 = \alpha - \alpha' s_1$ because $p_1$ is the cokernel of $\lambda$.  Now $j_1 \colon S_1' \to \overline{U_1}$ is a $\T_T$-preenvelope so there exists $\beta' \colon \overline{U_1} \to U$ such that $\beta = \beta' j_1$.  So $\alpha = \beta'j_1p_1 + \alpha's_1 = \begin{pmatrix}\alpha' & \beta' \end{pmatrix}t_1$.  This shows that $\alpha$ is in the image of $t_1^*$. The argument to show that $t_{i+1}^*$ is surjective for $1\leq i <n$ is similar with $\lambda_i$ playing the role of $\lambda$, $j_i$ playing the role of $\gamma$ and $j_i'$ playing the role of $s_1$.
\end{proof}

\begin{corollary}\label{Cor: diagram Frob}
Suppose we are in Setup \ref{SiltSet}. Then there is a commutative diagram in $\E$ with exact rows

\begin{equation}\label{SiltingDiagram}
    \xymatrix{
0 \ar[r] & A \ar[r] \ar@{=}[d] & T_0
^{(1)} \ar[r] \ar[d]^{f_0^{(1)}} & T_0^{(2)} \ar[r] \ar[d]^{f_0^{(2)}} & \dots \ar[r] & T_0^{(n)}  \ar[d]^{f_0^{(n)}} \ar[r] & 0 \\ 
0 \ar[r] & A \ar[r] \ar@{=}[d] & T_1^{(1)} \ar[r] \ar[d]^{f_1^{(1)}} & T_1^{(2)} \ar[r] \ar[d]^{f_1^{(2)}} & \dots \ar[r] & T_1^{(n)}  \ar[d]^{f_1^{(n)}} \ar[r] & 0\\
0 \ar[r] & A \ar[r] \ar@{=}[d] & T_2^{(1)} \ar[r] \ar[d]^-{f_2^{(1)}} & T_2^{(2)} \ar[r] \ar[d]^-{f_2^{(2)}} & \dots \ar[r] & T_2^{(n)}  \ar[d]^-{f_2^{(n)}}\ar[r] & 0 \\
&\vdots & \vdots &\vdots & &\vdots}
\end{equation}
 such that $f_{i}^{(j)}$ is a special $\T_{i+1}$-preenvelope for each $i\geq 0$ and $1 \leq j \leq n$.
    
\end{corollary}

Note that after passing to the image of the Diagram \ref{SiltingDiagram} in $\D(A)$ each $f_i^{(j)}$ becomes a $\U_{i+1}$-preenvelope corresponding to the co-t-structure $(\W_{i+1}, \U_{i+1})$.

The aim of the rest of this section is to prove the following theorem.

\begin{theorem}\label{ThmSilt}
    Suppose we are in Setup \ref{SiltSet} and consider the diagram in Corollary \ref{Cor: diagram Frob}. 
Consider the directed system
\begin{equation}\label{SeqForColim}
    T_0' \overset{f_0}{\longrightarrow} T_1' \overset{f_1}{\longrightarrow}T_2' \overset{f_2}{\longrightarrow} \dots
\end{equation}
with $T_i' := \oplus_{j=1}^n T_i^{(j)}$ and $f_i \colon T_i' \to T_{i+1}'=\oplus_{j=1}^n f_i^{(j)}$. 
Let  $T := \clim_{\mathbb{N}}  T'_i$, then $T\simeq \hclim_{\mathbb{N}}  T'_i$ is an $n$-silting object in $\D(A)$ with $\U_T=\cap_{i\geq 0}\U_i$.

\end{theorem}

Note that in $\D(A)$ we have $\Add{T'_i}=\Add{T_i}$, so we replace the $n$-silting objects $T_i$ by equivalent ones to construct the directed system.
 Before we prove the theorem we will need the following lemma.

\begin{lemma}\label{Ext-projective}
   Let $\U=\cap_{i\geq 0}\U_i$.
   Then for any object $X\in \U$ we have $\Hom_{\D(A)}(T, X[>0])=0$.
\end{lemma}
\begin{proof}
Applying $\Hom_{\D(A)}(-, X)$ to the Sequence \ref{SeqForColim}, we obtain an inverse system of abelian groups 
    \[
    \dots \overset{f_2^*}{\longrightarrow} \Hom_{\D(A)}(T_2', X) \overset{f_1^*}{\longrightarrow}\Hom_{\D(A)}(T_1', X) \overset{f_0^*}{\longrightarrow}\Hom_{\D(A)}(T_0', X)  
    \]
with each $f_i^*$ surjective because $f_i$ is a $\U_{i+1}$-preenvelope and $X \in  \U$ by assumption.  It follows that this inverse system satisfies the Mittag-Leffler condition and so we have an exact sequence of abelian groups (see, for example, \cite[Lemma 3.6]{GobelTrlifaj}):
    \begin{equation}\label{Eq: ML sequence} 0 \to \varprojlim_{i\geq 0}\Hom_{\D(A)}(T_i', X) \to \prod_{i\geq 0}\Hom_{\D(A)}(T_i', X) \xrightarrow{id - (f_i^*)} \prod_{i\geq 0}\Hom_{\D(A)}(T_i', X) \to 0. \end{equation}

By \cite[Proposition 11.3]{KellerNicolas} $T$ is isomorphic to the Milnor colimit of the Sequence \ref{SeqForColim}, so there exists a triangle 
    \[\coprod_{i\geq 0} T_i' \xrightarrow{id-(f_i)} \coprod_{i\geq 0} T_i' \to T \to \coprod_{i\geq 0} T_i'[1]\]
expressing $T$ as this Milnor colimit.  Applying $\Hom_{\D(A)}(-, X)$ to this triangle we obtain a long exact sequence 
\[
\resizebox{1\textwidth}{!}{
     $ \dots \to \prod_{i\geq 0}\Hom_{\D(A)}(T_i', X) \overset{x}{\to} \prod_{i\geq 0}\Hom_{\D(A)}(T_i', X) \to \Hom_{\D(A)}(T, X[1]) \to \prod_{i\geq 0}\Hom_{\D(A)}(T_i', X[1]) \rightarrow \dots .$}\]
By assumption we have that  $\prod_{i\geq 0}\Hom_{\D(A)}(T_i', X[n]) =0$ for all $n>0$ and so the exactness of the sequence implies that $\Hom_{\D(A)}(T, X[n]) =0$ for all $n >1$.  Moreover, the morphism $x = id - (f_i^*)$ is an epimorphism by Equation (\ref{Eq: ML sequence}) and so also $\Hom_{\D(A)}(T, X[1]) =0$.
\end{proof}

\begin{proof}[Proof of Theorem \ref{ThmSilt}]
Since colimits commute with finite coproducts, $T\simeq \oplus_{i=1}^nT^{(j)}$ with $T^{(j)}=\clim_{\mathbb{N}}  T^{(j)}_i$, where the colimit is taken along the maps $f_i^{(j)}$ for a fixed index $j$.
By exactness of directed colimits  in $\Ch(A)$, the diagram \ref{SiltingDiagram} gives an exact sequence 
\[
0 \to A \to T^{(1)} \to T^{(2)} \to \dots \to T^{(n)} \to 0
\]
in $\Ch(A)$. This sequence is spliced from short exact sequences, which gives a sequence of triangles in $\D(A)$:
\[
A \longrightarrow T^{(1)} \longrightarrow P ^{(1)} \longrightarrow A[1],
\]
\[
P ^{(1)} \longrightarrow T^{(2)} \longrightarrow P ^{(2)} \longrightarrow P ^{(1)}[1],\]
\[\dots\]
\[
P ^{(n-2)} \longrightarrow T^{(n-1)} \longrightarrow T ^{(n)} \longrightarrow P ^{(n-2)}[1].
\]
Hence $T$ is a generator of $\D(A)$.

The object $T$ belongs to $\U_i$ for all $i$. Indeed, $T$ is isomorphic to the homotopy colimit of the Sequence \ref{SeqForColim} with the first $i$ terms removed, since this gives a cofinal subsequence. This whole sequence belongs to $\U_i$ and aisles in $\D(A)$ are closed under directed homotopy colimits. So $T\in\U=\cap_{i\geq 0}\U_i$.

Consider the t-structure $({}^{\perp}\V_T,\V_T)$ generated by the object $T$, this t-structure exists by \cite{AJSS}. Let us denote ${}^{\perp}\V_T$ by $\U'$. By \cite[Remark 3]{NSZ} in order to prove that $T$ is silting it is enough to check that $T$ is a generator of $\D(A)$ and $\Hom_{\D(A)}(T, \U'[1])=0$. The subcategory $\U'$ is the smallest subcategory containing $T$ and closed under coproducts, positive shifts and extensions. 
Since $\U$ is also closed under all these operations and contains $T$, we get that $\U'\subseteq \U$. The claim $\Hom_{\D(A)}(T, \U'[1])=0$ follows from Lemma \ref{Ext-projective} and we get that $T$ is a silting object. In particular $\U'={}^{\perp}\V_T=\U_T$. 

From Lemma \ref{Ext-projective} we immediately get that $\U\subseteq \U_T$. So $\U_T=\U=\cap_{i\geq 0}\U_i$.

It remains to check that $T$ is an $n$-silting object.
By \cite[Proposition 4.17]{PV} and \cite[Lemma 4.5]{AMVSiltMod}, it is enough to check that $(\U_T,\V_T)$ is $n$-intermediate. Since each $(\U_i,\V_i)$ is $n$-intermediate, that is $\D^{\leq -n+1}\subseteq \U_i \subseteq \D^{\leq 0}$, the same holds for $\U$. 
\end{proof}

Note that, since the object $T$ constructed in Theorem \ref{ThmSilt} is shown to be an n-silting complex, it is isomorphic in $\D(A)$ to a complex of projective $A$ modules concentrated in degrees  $[-n+1,0]$.

\begin{remark}
     Starting from Setup \ref{SiltSet} the construction of the silting object $T$ carried out in this section is not unique up to isomorphism, however, it is unique up to equivalence of silting objects, so for any other $T'$ constructed using the same procedure we have $\Add(T)=\Add(T')$. 
\end{remark}

\section{Countable limits of cosilting objects} 

In this section we will consider the dual question of describing the cosilting object corresponding to the intersection of a sequence of nested coailes given by $n$-cosilting objects. The proofs can't be dualised literally because homotopy limits and Milnor limits require more care in $\D(A)$ but the strategy is very similar.
As before we will work in the following setting: $(Q,W,R)$ is a hereditary abelian model structure on $\Ch(A)$ whose homotopy category is $\D(A)$, $Q\cap R =\E$ and $\underline{\E}\simeq \D(A)$. One particular case where computations are easier is the injective model structure with $\E=\dgI$, since $n$-cotilting objects are isomorphic to bounded complexes of injectives. 

 We will use the following notation: for an $n$-cosilting object $C$ in $\D(A)$, we denote the cotorsion pair in $\E$ associated to the co-t-structure $(\V^C, \W^C)$ via the bijection from \cite[Proposition 3.16]{SS} by $(\F^C, \T^C)$. In particular, $\V^C=\underline{\F}^C[-1]$ and $\W^C=\underline{\T}^C$. 
For a sequence of $n$-cosilting objects $C_0, C_1, C_2, \dots \in \D(A)$ we will denote the co-t-structures $(\V^{C_i}, \W^{C_i})$ by $(\V_i, \W_i)$ and the cotorsion pairs $(\F^{C_i}, \T^{C_i})$ by $(\F_i, \T_i)$.

The aim of this section is to explicitly construct an $n$-cosilting object corresponding to the t-structure whose coaisle is the intersection of a countable sequence of decreasing coaisles corresponding to $n$-cosilting objects $C_0, C_1, C_2, \dots \in \D(A)$. So we will work in the following setup:

\begin{setup}\label{CoSiltSet}
    Let $C_0, C_1, C_2, \dots \in \D(A)$ be a sequence of $n$-cosilting objects such that $C_{i+1}\in \V_i$ for $i\geq 0$. Note that the last condition is equivalent to $\V_{i+1}\subseteq \V_i$, since $\V_{i+1}$ is the smallest coaisle in $\D(A)$ containing $C_{i+1}$.
\end{setup}

Recall that $E$ denotes a fixed injective cogenerator of $\Mod A$. For the rest of the section, by abuse of notation, $E$ will also denote a fixed fibrant-cofibrant replacement of $E$ in the model structure $(Q, W, R)$. The proof of the following proposition is dual to the proof of Proposition \ref{SiltApp} and is left to the reader.

\begin{proposition}
Suppose $B$ and $C$ are $n$-cosilting objects in $\D(A)$ such that $B \in \V^C$.   Then there exists a commutative diagram with exact rows in $\E$

\[\xymatrix{0 \ar[r]  & B_n \ar[r] \ar[d] & B_{n-1} \ar[r] \ar[d] & \dots \ar[r] & B_1 \ar[r] \ar[d] & E \ar[r] \ar@{=}[d] & 0  \\
0 \ar[r]  & C_n \ar[r]  & C_{n-1} \ar[r]  & \dots \ar[r] & C_1 \ar[r] & E \ar[r] & 0}\] 

such that, for every $1\leq i\leq n$, we have $C_i\in \F^C\cap\T^C$, $B_i\in\F^B\cap \T^B$ and the vertical arrows are special $\F^B$-precovers.
\end{proposition}

\begin{corollary}\label{Cor: co diagram Frob}
Suppose we are in Setup \ref{CoSiltSet}. Then there exists a commutative diagram in $\E$ with exact rows

\begin{equation}\label{eq: exact tower} \xymatrix{
&\vdots \ar[d]^{f_2^{(n)}} & \vdots  \ar[d]^{f_2^{(n-1)}} & \dots & \vdots \ar[d]^{f_2^{(1)}}  &\vdots \ar@{=}[d] & \\
0 \ar[r] & C_2^{(n)} \ar[r] \ar[d]^{f_1^{(n)}} &  C_2^{(n-1)} \ar[r] \ar[d]^{f_1^{(n-1)}} &   \dots \ar[r] & C_2^{(1)} \ar[r] \ar[d]^{f_1^{(1)}} & E  \ar@{=}[d] \ar[r] & 0 \\ 
0 \ar[r] & C_1^{(n)} \ar[r] \ar[d]\ar[d]^{f_0^{(n)}} &  C_1^{(n-1)} \ar[r] \ar[d]^{f_0^{(n-1)}} &   \dots \ar[r] & C_1^{(1)} \ar[r] \ar[d]^{f_0^{(1)}} & E  \ar@{=}[d] \ar[r] & 0 \\ 
0 \ar[r] & C_0^{(n)} \ar[r]  &  C_0^{(n-1)} \ar[r] &   \dots \ar[r] & C_0^{(1)} \ar[r]  & E   \ar[r] & 0  
}
\end{equation} such that $f_{i}^{(j)}$ is a special $\F_{i+1}$-precover for each $i\geq 0$ and $1 \leq j \leq n$.
    
\end{corollary}

Considering the image of the diagram \ref{eq: exact tower} in $\D(A)$  each $f_i^{(j)}$ gives a $\V_{i+1}$-precover corresponding to the co-t-structure $(\V_{i+1}, \W_{i+1})$.

\begin{theorem}\label{ThmCoSilt}
    Suppose we are in Setup \ref{CoSiltSet} and consider the diagram in Corollary \ref{Cor: co diagram Frob}. Consider the inverse system 
\begin{equation}\label{SeqForLim}
    \dots \overset{f_2}{\longrightarrow} C_2'\overset{f_1}{\longrightarrow} C_1'\overset{f_0}{\longrightarrow} C_0' 
\end{equation}
 with $C_i' := \oplus_{j=1}^n C_i^{(j)}$ and $f_i \colon C_{i+1}' \to C_{i}'=\oplus_{j=1}^n f_i^{(j)}$. Let $C := \lim_{\mathbb{N}^{op}}  C_i$. Then $C\simeq \hlim_{\mathbb{N}^{op}}  C_i$ is an $n$-cosilting object in $\D(A)$ with $\V^C=\cap_{i\geq 0}\V_i$.
\end{theorem}

In $\D(A)$, we have $\Prod C_i'=\Prod C_i$, so we only replace $n$-cosilting objects $C_i$ by equivalent
ones to construct the inverse system. Here the limit is taken along the maps $\oplus_{i=1}^n f_{i}^{(j)}$. Since each object $C_i^{(j)}$ belongs to $\E$ and the maps $f_{i}^{(j)}$ are surjective with kernel in $\E$, we see by discussion in the end of Section \ref{SecHomLimColim}, that the diagram is fibrant and the homotopy limit in $\D(A)$ can be computed as limit in $\Ch(A)$. 
 Before we prove the theorem we will need a couple of lemmas.

\begin{lemma}\label{MilnorLim}
Object $C$ is isomorphic to the Milnor limit of the Sequence \ref{SeqForLim}. In particular, there is a triangle 
  \[C \to \prod_{i\geq 0} C_i' \xrightarrow{id-(f_i)} \prod_{i\geq 0} C_i' \to C[1].\]
\end{lemma}
\begin{proof}
Since all morphisms $f_i$ are surjective this inverse system satisfies the Mittag-Leffler condition in $\Ch(A)$ and so by \cite[Lemma 3.6]{GobelTrlifaj} we have a short exact sequence in $\Ch(A)$:
\[ 0 \to \varprojlim_{i\geq 0} C_i' \to \prod_{i\geq 0} C_i' \xrightarrow{id - (f_i)} \prod_{i\geq 0} C_i' \to 0. \]
 This short exact sequence in $\Ch(A)$ gives a triangle in $\D(A)$ representing $C$ as the Milnor limit of Sequence \ref{SeqForLim}.
\end{proof}

\begin{lemma}\label{Ext-injective}
   Let $\V:=\cap_{i\geq 0}\V_i$.
   Then for any object $X\in \V$ we have $\Hom_{\D(A)}(X, C[>0])=0$.
\end{lemma}
\begin{proof}
Applying $\Hom_{\D(A)}(X, -)$ to the Sequence \ref{SeqForLim}, we obtain an inverse system of abelian groups 
    \[
    \dots \overset{f_2^*}{\longrightarrow} \Hom_{\D(A)}(X, C_2') \overset{f_1^*}{\longrightarrow}\Hom_{\D(A)}(X, C_1') \overset{f_0^*}{\longrightarrow}\Hom_{\D(A)}(X, C_0')  
    \]
with each $f_i^*$ surjective because $f_i$ is a $\V_{i+1}$-precover and $X \in \V$ by assumption.  It follows that this inverse system satisfies the Mittag-Leffler condition and so again by \cite[Lemma 3.6]{GobelTrlifaj} we have an exact sequence of abelian groups:
    \begin{equation}\label{Eq: ML sequence 2} 0 \to \varprojlim_{i\geq 0}\Hom_{\D(A)}(X, C_i') \to \prod_{i\geq 0}\Hom_{\D(A)}(X, C_i') \xrightarrow{id - (f_i^*)} \prod_{i\geq 0}\Hom_{\D(A)}(X, C_i') \to 0. \end{equation}

Now consider the triangle 
    \[C \to \prod_{i\geq 0} C_i' \xrightarrow{id-(f_i)} \prod_{i\geq 0} C_i' \to C[1]\]
whose existence is guaranteed by Lemma \ref{MilnorLim}.  Applying $\Hom_{\D(A)}(X,-)$ to this triangle we obtain a long exact sequence
\[
\resizebox{1\textwidth}{!}{
     $\dots \to \prod_{i\geq 0}\Hom_{\D(A)}(X, C_i') \overset{x}{\to} \prod_{i\geq 0}\Hom_{\D(A)}(X, C_i') \to \Hom_{\D(A)}(X, C[1]) \to \prod_{i\geq 0}\Hom_{\D(A)}(X, C_i'[1]) \rightarrow \dots .$}\]
By assumption we have that  $\prod_{i\geq 0}\Hom_{\D(A)}(X, C_i'[n]) =0$ for all $n>0$ and so the exactness of the sequence implies that $\Hom_{\D(A)}(X, C[n]) =0$ for all $n >1$.  Moreover, the morphism $x = id - (f_i^*)$ is an epimorphism by Equation (\ref{Eq: ML sequence 2}) and so also $\Hom_{\D(A)}(X, C[1]) =0$.
\end{proof}

\begin{proof}[Proof of Theorem \ref{ThmCoSilt}]
To prove that the object $C$ is cosilting we will check the conditions of \cite[Theorem 2.8]{Breaz}. Concretely, we need to check that 
\begin{enumerate}
    \item $\Hom_{\D(A)}(\Prod(C),C[>0])=0$;  
    \item $C$ is a cogenerator of $\D(A)$;
    \item there exists a complete precoaisle $\V'$ containing $C$;
    \item $\Hom_{\D(A)}(\V'[-m], C)=0$ for some $m>0$.
\end{enumerate}

Here by a \emph{complete precoaisle}  in $\D(A)$ we mean a subcategory of $\D(A)$ closed under extensions, products, direct
summands and negative shifts.
The role of $\V'$ will be played by $\V$, which is a complete precoaisle, since each of $\V_i$ is. Let us check that $C\in \V$. Indeed, $\V$ is closed under homotopy limits as each of $\V_i$ is (see \cite[Proposition 5.2]{SSV}), and $C$ is isomorphic to the homotopy limit of the Sequence \ref{SeqForLim} with the first $i$ terms removed for each $i$. This shows (3). Condition (1) follows from Lemma \ref{Ext-injective} and the fact that $\V$ is closed under products. Condition (4) follows from Lemma \ref{Ext-injective} again with $m=1$.

The proof of the fact that $C$ is a cogenerator in $\D(A)$ is dual to the corresponding proof for the silting case. Diagram \ref{eq: exact tower} is spliced from short exact sequences in $\Ch(\Mod A^{\mathbb{N}^{op}})$ by construction. Taking the homotopy limits of these short exact sequences gives $n-1$ triangles in $\D(A)$, inductively building $E$ from summands of $C$. This shows that $E$ belongs to the thick subcategory generated by $C$, so $C$ cogenerates $\D(A)$. This proves (4), and we get that $C$ is cosilting.

Let us check that the coaisle $\V^C$ coincides with $\V=\bigcap_{i\geq 0}\V_i$. 
    Since $\V^C$ is the smallest coaisle in $\D(A)$ containing $C$ and $C\in \V_i$ for all $i$ we get that $\V^C \subseteq \V_i$ for all $i$, so $\V^C \subseteq \V$. Now by Lemma \ref{Ext-injective} again we get  $\V\subseteq \V^C$. So $\V^C=\V$.

  By \cite[Proposition 4.17]{PV} a cosilting object in $\D(A)$ is $n$-cosilting if and only if the corresponding t-structure is $n$-intermediate, which holds for $(\U^C,\V^C)$. 
\end{proof}

We will finish this section with the following observation: under certain technical assumptions the $\F$-precover associated to the cotorsion pair $(\F, \T )$ can be chosen to be an $\F$-cover as follows from the following proposition. The conditions of the proposition hold in the case when the co-t-structure is associated to an $n$-cosilting object in $\D(A)$ \cite[Theorem 4.6]{LakingDer}. We will prove the proposition for the injective model structure, which is the most relevant for the case of cosilting objects.

\begin{proposition}
    Let $(\F, \T )$ be a cotorsion pair in $\dgI$ corresponding to a co-t-structure $(\underline{\F}[-1], \underline{\T} )$ in $\D(A)$ such that $\underline{\F}[-1]$ is closed under directed homotopy colimits, then for any $X\in \dgI$ the cotorsion pair approximation sequence 
    \[
    0\rightarrow T\rightarrow F \xrightarrow{f} X \rightarrow 0
    \]
    with $T\in\T$ and $F\in\F$ can be chosen in such a way that $f$ is an $\F$-cover.
\end{proposition}

\begin{proof}
    Let us denote by $\overline{\F}$ the preimage in $\Ch (A)$ of $\underline{\F}$ under the localisation $\Ch(A)\rightarrow \D(A)$. Note that $\overline{\F}$ consists of complexes whose fibrant replacement belongs to $\F$. 
    
    Let us first prove that $\overline{\F}$ is precovering in $\Ch(A)$. For any $X\in \Ch(A)$ one can consider the approximation sequence 
    \[
    0 \rightarrow X \xrightarrow{x} iX \rightarrow aX \rightarrow 0
    \]
    induced from the cotorsion pair $(\Ac,\dgI)$ and the map $iW\xrightarrow{\epsilon} iX\rightarrow 0$, with $iW \in \F$ given by the cotorsion pair $(\F,\T)$ in $\dgI$, where $iW$ is just the notation for some object of $\F$. Taking the pullback of $x$ and $\epsilon$ we get the diagram in $\Ch(A)$:

   \[\xymatrix{0 \ar[r] & W \ar[r]^{w} \ar[d]^{\tau} & iW \ar[r]^{} \ar[d]^{\epsilon} & aX \ar[r] \ar@{=}[d]  & 0  \\
0 \ar[r] & X \ar[r]^{x}  & iX \ar[r]^{}  & aX \ar[r]  & 0.}\]

Since $\epsilon$ is an epimorphism, the pullback square is a pushout square and the cokernel of $w$ is isomorphic to $aX$ justifying the notation. We see that $W\in \overline{\F}$. Let us check that $\tau$ is an  $\overline{\F}$-precover of $X$. For any map $g : W' \rightarrow X$ with $W'\in \overline{\F}$ we can consider the following diagram, where $w'$ is a $\dgI$-precover of $W'$:
  \[\xymatrix{ 
   W' \ar[r]^{w'} \ar@/_1pc/[ddr]^{g} \ar[dr]^{\kappa} & iW'  \ar@/^3pc/[ddr]^{t} \ar[dr]^{\gamma}  \\
 &  W \ar[r]^{} \ar[d]^{\tau} & iW  \ar[d]_{\epsilon}   \\
& X \ar[r]^{x}  & iX.   }\]
The map $t$ exists since $w'$ is a $\dgI$-precover and gives $tw'=xg$. The map $\gamma$ exists since $\epsilon$ is an $\F$-precover and gives $\epsilon \gamma= t$. The map $\kappa$ exists from the properties of the pullback and gives $g=\tau\kappa$. Proving that $\overline{\F}$ is precovering in $\Ch(A)$.  
    
    By \cite[Theorem 1.2]{ElBashir}, the class $\overline{\F}$ is covering in $\Ch (A)$ since it is closed under directed colimits (which follows from $\underline{\F}$ being closed under homotopy colimits) and precovering. 

    Let us now check that $\F$ is covering in $\dgI$. Assume $X\in \dgI$ and let
    \[
    0\rightarrow T'\rightarrow F' \xrightarrow{f'} X \rightarrow 0
    \]
    be the approximation sequence in $\dgI$ existing by completeness of $(\F, \T )$ and let $F \xrightarrow{f} X$ be the $\overline{\F}$-cover of $X$ with respect to $\overline{\F}$. Since $\F\subseteq \overline{\F}$, the morphism $f'$ factors through $f$ and $f$ is surjective. Note that $f'$ is an $\F$-precover in $\dgI$.

    Additionally, let us consider the approximation sequence $0 \rightarrow F \xrightarrow{i} iF \rightarrow aF \rightarrow 0 $, where $iF\in \dgI$ and $aF$ is acyclic, arising from the complete cotorsion pair $(\Ac,\dgI)$ in $\Ch(A)$. The morphism $F \xrightarrow{i} iF$ is a $\dgI$-preenvelope of $F$ with $iF\in \F$.

    \[\xymatrix{ &  F\ar[r]^{i} \ar[d]^{f} & iF \ar[dl]^{h} \ar@/_/[dll]_{w} \\
                F' \ar[r]^{f'} & X & }\] 

    The map $f$ factors as $hi$ since $X\in\dgI$, and the map $h$ factors as $f'w$, since $iF\in \F$. So $f=f'wi$. On the other hand $f'=fg$ for some $g: F'\rightarrow F$. We get $f=fgwi$. By minimality of $f$, the morphism $gwi$ is an isomorphism and $F$ is a direct summand of $F'$. Note that $F\in \dgI$ since $\dgI$ is closed under direct summands.

    By the universality of kernels we get that the kernel $T$ of $f$ is a direct summand of $T'$, so it belongs to $\T$ and $0\rightarrow T\rightarrow F \xrightarrow{f} X \rightarrow 0$ is the desired approximation sequence.
    \end{proof}

\section{Continuous colimits of silting objects}\label{SecContSilt}

In this section we will consider the situation similar to Section \ref{SecSilt} with the family of $n$-silting objects indexed not by $\mathbb{N}$ but by an arbitrary ordinal $\mu$. 

To make the results work for this more general type of diagrams we will need to make additional assumptions guaranteeing that the colimits taken as intermediate steps of the construction remain in the class of bifibrant objects. This will work, for example, if we consider the projective model structure in the category $\Ch(A)$ over a right perfect ring $A$, which ensures that any direct limit of projective modules is projective. This also works for the injective model structure over a right noetherian ring, which ensures that any direct limit of injective modules is injective.

Let $\mu$ be an ordinal. A direct system $((T_\alpha)_{\alpha<\mu},(f_{\alpha\beta})_{\alpha<\beta<\mu})$ is called a \emph{continuous direct $\mu$-system}
if, for each limit ordinal $\gamma<\mu$, we have $T_\gamma=\clim_{\alpha < \gamma } T_\alpha$ and $(f_{\alpha\gamma})_{\alpha<\gamma}$ are the canonical morphisms guaranteed by the definition of the colimit. If the colimit of the whole direct system exists, the canonical
morphism $T_0 \rightarrow \clim_{\alpha < \mu } T_\alpha$ is called the \emph{transfinite composition} of the direct $\mu$-system.

We will use the following version of the Eklof's lemma \cite[Lemma 9.3]{gillespie}. It holds in any exact category, in particular, in the category of chain complexes.

\begin{proposition}\cite[Lemma 9.3]{gillespie}
    Let $((T_\alpha)_{\alpha<\mu},(f_{\alpha\beta})_{\alpha<\beta<\mu})$ be a continuous direct $\mu$-system in $\Ch (A)$ such that all the morphisms $f_{\alpha\alpha+1}$ are
 monomorphisms. Let $T$ denote $\clim_{\alpha < \mu } T_\alpha$. Then $\Ext^1_{\Ch(A)}
(T,Y) = 0$ for any 
$Y \in \Ch(A)$ such that for each $\alpha < \mu$ \[\Ext^1_{\Ch(A)}
(T_0,Y) = 0 \text{ and }\Ext^1_{\Ch(A)}
(\Coker (f_{\alpha\alpha+1}),Y) = 0.\] 
\end{proposition}

For the rest of the section we will use the notation $(\U_\alpha, \V_\alpha)$ for the t-structure $(\U_{T_\alpha},\V_{T_\alpha})$ and $(\F_\alpha, \T_\alpha)$ for the cotorsion pair $(\F_{T_\alpha},\T_{T_\alpha})$ for any ordinal $\alpha$. 

\begin{setup}\label{SetupSiltCont}
Let $A$ be right perfect or right noetherian. Let $\mu$ be a fixed limit ordinal. Let  $\{T_\alpha\}$ be a sequence of $n$-silting objects in $\D(A)$ parametrised by all non-limit ordinal $\alpha<\mu$ and such that $T_\beta \in \U_{\alpha}$ for any non-limit ordinals $\alpha,\beta$ with $\alpha<\beta<\mu$.    
\end{setup}

Let us construct a continuous direct $\mu$-system adding objects $T_\beta$ for limit ordinals $\beta$ and constructing the analogue of Diagram \ref{SiltingDiagram}. Assume that the diagram is constructed for all ordinals $\alpha<\beta$, that is for $\alpha<\beta$ we assume that 
\begin{enumerate}
    \item there exists an exact sequence in $\E=\dgP$ in case $A$ is right perfect (or in $\E=\dgI$ in case $A$ is right noetherian, respectively; in this case, as before, we will abuse the notation and denote the fibrant-cofibrant replacement of $A$ by $A$)
\[
0 \rightarrow  A \rightarrow T_{\alpha}
^{(1)} \rightarrow  T_{\alpha}^{(2)} \rightarrow \dots \rightarrow T_{\alpha}^{(n)}  \rightarrow 0,
\]
such that $T_{\alpha}^{(j)}\in \F_\alpha\cap \T_\alpha$ in case $\alpha$ is a non-limit ordinal;
\item there are maps $(f^{(j)}_{\alpha'\alpha} \colon T_{\alpha'} ^{(j)} \to T_{\alpha}^{(j)})_{\alpha'<\alpha<\beta}$ such that $f^{(j)}_{\alpha'\alpha'+1}$ is a special $\T_{\alpha'+1}$-preenvelope for each $1\leq j\leq n$ and $\alpha'+1<\beta$;

\item the equalities $f^{(j)}_{\alpha''\alpha}f^{(j)}_{\alpha'\alpha''}=f^{(j)}_{\alpha'\alpha}$ hold for all $\alpha'<\alpha''<\alpha<\beta$ and the squares formed by the horizontal maps and the maps $f^{(j)}_{\alpha'\alpha}$ for different values of $j$ commute; 
\item whenever $\alpha$ is a limit ordinal, we have $T_\alpha^{(j)}= \clim_{\alpha' < \alpha } T_{\alpha'}^{(j)}$ and $(f^{(j)}_{\alpha'\alpha})_{\alpha'<\alpha}$ are colimit maps;
\item the objects $T'_{\alpha}=\oplus_{j=1}^n T_{\alpha}^{(j)}$ 
are $n$-silting objects with $T'_\alpha \in \U_{\alpha'}$ for any $\alpha'< \alpha< \beta$ and $\Add(T'_\alpha)=\Add(T_\alpha)$ in case $\alpha$ is a non-limit ordinal.
\end{enumerate}

If $\beta$ is a successor ordinal, we proceed as in Corollary \ref{Cor: diagram Frob} applying Proposition \ref{SiltApp} to $S=T_{\beta-1}$ and $T=T_{\beta-1}$. The maps $f^{(j)}_{\alpha'\beta}$ are defined as $f^{(j)}_{\beta-1\beta}\circ f^{(j)}_{\alpha'\beta-1}$. In that case all the conditions listed above hold and the diagram is constructed for the ordinal $\beta$.

If $\beta$ is a limit ordinal, we define $T_\beta^{(j)}= \clim_{\alpha < \beta } T_\alpha^{(j)}$ and $(f^{(j)}_{\alpha\beta})_{\alpha<\beta}$ to be the colimit maps. All the necessary diagrams commute from the universal property of the colimit. In case $A$ is  right perfect each $T_\beta^{(j)}$ is an $n$-term complex of projective modules, so it belongs to $\E=\dgP$; if $A$ is right noetherian each $T_\beta^{(j)}$ is a bounded below complex of injective modules, so it belongs to $\E=\dgI$. The only condition left to check to finish the construction of the diagram for the ordinal $\beta$ is condition (5), which follows from the following lemma:

\begin{lemma}\label{TransFinIndStep}
    The object $T'_{\beta}=\oplus_{j=1}^n T_{\beta}^{(j)}$ is an n-silting object. Moreover, $\U_\beta=\cap_{\alpha<\beta}\U_\alpha$, in particular, $T'_{\beta} \in \U_{\alpha}$ holds for any $\alpha<\beta$.
\end{lemma}

\begin{proof}
   Let $\U$ be the intersection $\cap_{\alpha< \beta}\U_{\alpha}$. Recall that $\T_{\alpha}$ denotes the preimage of $\U_{\alpha}$ in the category $\E$ ($\E=\dgP$ or $\E=\dgI$, respectively). Let us denote the preimage of $\U$ by $\T=\cap_{\alpha< \beta}\T_{\alpha}$. We want to apply Eklof's lemma to the system $((T_{\alpha}^{(j)})_{\alpha<\beta}, (f^{(j)}_{\alpha'\alpha})_{\alpha'<\alpha<\beta})$ and any object $Y\in \T$.  
   All the maps $f_{\alpha\alpha+1}$ are monomorphisms in $\Ch(A)$ by construction, since they are admissible monos in the corresponding subcategory of bifibrant objects. Let $Y$ be any object in $\T$. We have $\Ext^1_{\Ch(A)}
(\Coker (f_{\alpha\alpha+1}),Y) = 0$ for each $\alpha < \beta$ since by construction $f_{\alpha\alpha+1}$ is a special $\T_{\alpha+1}$-preenvelope and $Y\in \T_{\alpha+1}$. The same holds for $T_0^{(j)}$ since $Y\in \T_0$. We get $\Ext^1_{\Ch(A)}
(T_{\beta}^{(j)},Y) = 0$.

Passing to $\D(A)$ we conclude that $\Hom_{\D(A)}(T'_{\beta}, Y[>0])=0$ for any $Y\in \U$, since \[\Hom_{\D(A)}(T'_{\beta}, Y[1])=\Ext^1_{\Ch(A)}
(T'_{\beta},Y)=0\] and $\U$ is closed under positive shifts.
Replacing the use of Lemma \ref{Ext-projective} in the proof of Theorem \ref{ThmSilt} by the above argument we get that $T'_{\beta}$ is an $n$-silting object whose aisle $\U_\beta$ coincides with $\U=\cap_{\alpha<\beta}\U_\alpha$.
\end{proof}

By transfinite induction and by Lemma \ref{TransFinIndStep} we get the following theorem.

\begin{theorem}\label{contcolim}
   Suppose we are in Setup \ref{SetupSiltCont}. Then there exists a continuous direct $\mu$-system $((T'_\alpha)_{\alpha<\mu},(f_{\alpha\beta})_{\alpha<\beta<\mu})$ in $\Ch (A)$ with $\Add(T_\alpha)=\Add(T'_\alpha)$ for all non-limit ordinals $\alpha<\mu$ and such that $T_{\mu}=\clim_{\alpha < \mu } T'_\alpha$ is an $n$-silting object in $\D(A)$ with $T_\mu\simeq \hclim_{\alpha < \mu } T'_\alpha$ and $\U_{T_\mu}=\cap_{\alpha<\mu}\U_\alpha$.
\end{theorem}

\section{Continuous limits of cosilting objects of cofinite type}\label{SecCoCont}
In this section we wish to dualise the results of Section \ref{SecContSilt}. Our strategy is to use a duality between $\D(A)$ and $\D(A^{op})$ that will allow us to use the results of the previous section directly. We chose this strategy since to the best of our knowledge the subcategory of bifibrant objects is rarely closed under appropriate inverse limits. Choosing a commutative ring $k$ such that $A$ is a $k$-algebra and an injective cogenerator $W$ of $\Mod k$ (e.g.~$k=\mathbb{Z}$, $W=\mathbb{Q}/\mathbb{Z}$), we consider the exact functor $(-)^+ \colon \Ch(A) \to \Ch(A^{op})$ obtained by applying the functor $\mathrm{Hom}_k(-, W)$ pointwise. We use the same notation for the corresponding derived functor 
    \[(-)^+ := \mathbf{R}\mathrm{Hom}_k(-, W) \: \colon\: \D(A) \to \D(A^{op})\]
and for the analogous functor in the other direction $(-)^+ := \mathbf{R}\mathrm{Hom}_k(-, W) \: \colon\: \D(A^{op}) \to \D(A)$. 

We say that an n-cosilting complex $C$ is of \emph{cofinite type} if there is a set $\mathcal{S}$ of compact objects in $\D(A)$ such that $\mathcal{V}^C = \mathcal{S}^{\perp}$. 
Let us summarise the results of \cite{AHpara}, which we are going to use, see \cite[Sec.~7]{BirdWill} for a more general treatment. The fact that the bijection between bounded silting and cosilting complexes of cofinite type restricts to complexes of projectives and injectives concentrated in $n$ degrees follows since $\mathrm{Hom}_k(-, W)$ is exact and sends projective modules to injective modules.

\begin{theorem}\label{BijCofiniteType}\cite[Lemma 2.3, Lemma 2.5, Theorem 3.1, Theorem 3.3]{AHpara}
Let $A$ be a ring. The duality $(-)^+$ induces a bijection between equivalence classes of $n$-silting complexes in $\D(A^{op})$  and equivalence classes of $n$-cosilting complexes of cofinite type in $\D(A)$. Moreover, the aisle and the coaisle of the corresponding t-structures are related as follows: Let $T \in \D(A^{op})$ be an $n$-silting complex, then
\begin{enumerate}
    \item For every $X\in \D(A^{op})$ we have that $X\in \U_T$ if and only if $X^+ \in \V^{T^+}$.
    \item For every $Y\in \D(A)$ we have that $Y\in \V^{T^+}$ if and only if $Y^+ \in \U_T$.
\end{enumerate}
\end{theorem}

Let $\mu$ be an ordinal. Recall that a $\mu^{op}$-shaped diagram $((C_\alpha)_{\alpha<\mu},(g_{\beta\alpha})_{\alpha<\beta<\mu})$ is a \emph{continuous inverse $\mu$-system} if, for each limit ordinal $\gamma<\mu$, we have $C_\gamma = \lim_{\alpha < \gamma } C_\alpha$ and the morphisms $(g_{\gamma\alpha})_{\alpha < \gamma}$ are the canonical morphisms given by the limit.

For the rest of the section we will use the notation $\V_\alpha$ for the class $\V^{C_\alpha}$ for any ordinal $\alpha$.

\begin{setup}\label{SetupCosiltCont}
Let $A$ be left perfect. Let $\mu$ be a fixed limit ordinal. Let  $\{C_\alpha\}$ be a sequence of $n$-cosilting objects of cofinite type in $\D(A)$ parametrised by all non-limit ordinals $\alpha<\mu$ and such that $C_\beta \in \V_{\alpha}$ for any non-limit ordinals $\alpha,\beta$ with $\alpha<\beta<\mu$.    
\end{setup}

\begin{theorem}\label{ThmCoCont}
   Suppose we are in Setup \ref{SetupCosiltCont}. Then there exists a continuous inverse $\mu$-system $((C'_\alpha)_{\alpha<\mu},(g_{\beta\alpha})_{\alpha<\beta<\mu})$ in $\Ch (A)$ with $\Prod(C_\alpha)=\Prod(C'_\alpha)$ for all non-limit ordinals $\alpha<\mu$ and such that $C_{\mu}:=\lim_{\alpha < \mu } C'_\alpha$ is an $n$-cosilting object of cofinite type in $\D(A)$ with $C_{\mu}\simeq \hlim_{\alpha < \mu } C'_\alpha$ and $\V^{C_\mu}=\cap_{\alpha<\mu}\V_\alpha$.
\end{theorem}
\begin{proof} In the proof of this theorem we will work with the injective model structure in $\Ch(A)$ and the projective model structure in $\Ch(A^{op})$.
By Theorem \ref{BijCofiniteType}, there exists a sequence of silting objects $\{T_\alpha\}$ in $\D(A^{op})$ indexed by non-limit ordinals $\alpha<\mu$ with $T_\beta \in \mathcal{U}_\alpha$ for $\alpha <\beta <\mu$ such that $\Add (T_\alpha^+) = \Add (C_\alpha)$. 
 In other words, the set $\{T_\alpha\}$ satisfies the conditions of Setup \ref{SetupSiltCont} for the right perfect ring $A^{op}$. 
By Theorem \ref{contcolim}, there exists a continuous direct $\mu$-system $((T_\alpha')_{\alpha<\mu}, (f_{\alpha\beta})_{\alpha<\beta<\mu})$ in $\Ch(A^{op})$ such that $\Add(T_\alpha)= \Add(T_\alpha')$ for every non-limit ordinal $\alpha<\mu$ and $T_\mu= \clim_{\alpha<\mu} T_\alpha'$ is an $n$-silting object with $\U_{T_\mu} = \bigcap_{\alpha<\mu} \U_\alpha$.     
It follows directly from Theorem \ref{BijCofiniteType} that $C_\mu:=T_\mu^+$ is an $n$-cosilting object of cofinite type in $\D(A)$ with $\V^{T_\mu^+} = \cap_{\alpha<\mu}\V_\alpha$. We wish to show that $T_\mu^+$ is isomorphic to the limit described in the statement of the theorem.

We observe that $((C'_\alpha)_{\alpha<\mu},(g_{\beta\alpha})_{\alpha<\beta<\mu}):= ((T_\alpha'^+)_{\alpha<\mu}, (f^+_{\alpha\beta})_{\alpha<\beta<\mu})$ is a continuous inverse $\mu$-system in $\Ch(A)$ with $\Prod(C_\alpha) = \Prod(C_\alpha')$ for every non-limit ordinal $\alpha < \mu$. Continuity follows since $(-)^+$ sends colimits to limits in a canonical way. We also have that $\Prod(C_\alpha)=\Prod(C_\alpha')$ because $\Add(T_\alpha)= \Add(T_\alpha')$ and the bijection induced by $(-)^{+}$ sends equivalent silting objects to equivalent cosilting objects.

Finally we observe that $T_\mu^+ \cong \hlim_{\alpha<\mu} C_\alpha'$. Indeed, in the category of chain complexes, we have that $T_\mu^+ \cong  \lim_{\alpha<\mu} C_\alpha'$ because $(-)^+$ sends colimits to limits. Moreover, the construction of  $((T_\alpha')_{\alpha<\mu}, (f_{\alpha\beta})_{\alpha<\beta<\mu})$ implies that the continuous inverse $\mu$-system $((C'_\alpha)_{\alpha<\mu},(g_{\beta\alpha})_{\alpha<\beta<\mu})$ is such that each $C'_\alpha$ is an $n$-term complex of injectives and all the maps $g_{\alpha+1 \alpha }$ are epimorphisms with fibrant kernels (with respect to the injective model structure on $\Ch(A)$). In Section \ref{SecHomLimColim}, we saw that this implies that $((C'_\alpha)_{\alpha<\mu},(g_{\beta\alpha})_{\alpha<\beta<\mu})$ is a fibrant object of $\Ch(A)^{\mu^{op}}$. Therefore $\hlim_{\alpha<\mu} C_\alpha'$ coincides with $\lim_{\alpha<\mu} C_\alpha'$ considered as an object of $\D(A)$. 
\end{proof}

\section{Application: Numerical torsion pairs and tame algebras}\label{SecApp}

Let us start by briefly recalling the context in which the wall and chamber structure of the real Grothendieck group and the $g$-vector fan of a finite dimensional algebra are studied. 

In this section we will assume that $A$ is a basic finite-dimensional algebra over an algebraically closed field $k$. All silting complexes considered in this section will be two-term. The rank of the Grothendieck group of $A$ will be denoted by $l$. The isomorphism classes of indecomposable projective modules $P_1,\dots,P_l$ with $A\simeq \oplus_{i=1}^l P_i $ form a basis of the Grothendieck group $\Ko_0(\proj A)$ and the isomorphism classes of indecomposable simple modules $S_1,\dots,S_l$, where $S_i$ denotes the top of $P_i$, form a basis of the Grothendieck group $\Ko_0(\fmod A)$.  There is a $\mathbb{Z}$-bilinear form 
\[
\Ko_0(\proj A) \times \Ko_0(\fmod A) \rightarrow  \mathbb{Z}
\]
defined by $\langle P_i , S_j \rangle=\delta_{i,j}$, which is called the \emph{Euler form}. 

For simplicity throughout this section we will denote $\K^b(\proj A)$ by $\K^b(A)$ and $\D^b(\fmod A)$ by $\D^b(A)$, their Grothendieck groups will be denoted by $\Ko_0(\K^b(A))$ and $\Ko_0(\D^b(A))$, respectively. Traditionally the elements of $\Ko_0(\K^b(A))$ are called \emph{$g$-vectors}, so for an object $P\in \K^b(A)$ its class $[P]\in \Ko_0(\K^b(A))$ will be called the \emph{$g$-vector of $P$}. Since $\Ko_0(\proj A)\simeq \Ko_0(\K^b(A))$ and $\Ko_0(\fmod A)\simeq \Ko_0(\D^b(A))$ we get a $\mathbb{Z}$-bilinear form
\[
\Ko_0(\K^b(A)) \times \Ko_0(\D^b(A)) \rightarrow  \mathbb{Z}.
\]
One can show that
\[
\langle P , X \rangle = \sum_{n \in \mathbb{Z}} (-1)^n \dim_k \Hom_{\D^b(A)}(P,X[n]).
\]

We can consider the real Grothendieck groups $\Ko_0(\proj A) \otimes_\mathbb{Z} \mathbb{R} \simeq \mathbb{R}^l$ and $\Ko_0(\fmod A) \otimes_\mathbb{Z} \mathbb{R} \simeq \mathbb{R}^l$ and extend the Euler form $\mathbb{R}$-linearly to an $\mathbb{R}$-bilinear form
\[
[\Ko_0(\proj A) \otimes_\mathbb{Z} \mathbb{R} ] \times [\Ko_0(\fmod A) \otimes_\mathbb{Z} \mathbb{R}] \rightarrow  \mathbb{R}.
\]
Thus, to a vector $\theta \in \mathbb{R}^l$ one can associate the $\mathbb{R}$-linear form $\theta: \Ko_0(\fmod A) \otimes_\mathbb{Z} \mathbb{R} \rightarrow \mathbb{R}$ given by $\langle \theta , - \rangle$. Following \cite{BKT}, we can define two numerical torsion pairs for each $\theta \in \mathbb{R}^l$: $(\overline{\T}_\theta, \F_\theta)$ and $({\T}_\theta, \overline{\F}_\theta)$, where
\begin{align*} 
\overline{\T}_\theta & =  \{ X \in \fmod A | \theta(X') \geq 0, \text{ for all factor modules } X' \text{ of } X\}, \\ 
{\F}_\theta & =  \{ X \in \fmod A | \theta(X') < 0, \text{ for all submodules } X' \text{ of } X, X'\neq 0 \}, \\ 
{\T}_\theta & =  \{ X \in \fmod A | \theta(X') > 0, \text{ for all factor modules } X' \text{ of } X, X'\neq 0\}, \\ 
\overline{\F}_\theta & =  \{ X \in \fmod A | \theta(X') \leq 0, \text{ for all submodules } X' \text{ of } X\}.
\end{align*}

The intersection $\overline{\T}_\theta \cap \overline{\F}_\theta$ is a wide subcategory of $\fmod A$ of \emph{$\theta$-semistable objects}, that is modules $X$ such that $\theta(X) = 0$ and for all factor modules $X'$ of $X$, we have $\theta(X') \geq 0$. Note that for any $\epsilon \in \mathbb{R}, \epsilon >0$ the torsion pairs corresponding to  $\theta$ and $\epsilon \theta$ coincide. Following \cite{AsaiIyama}, we write 
\[
\theta \geq \eta \text{ if } \theta-\eta \in \mathbb{R}^l_{\geq 0}=\sum_{i=1}^l\mathbb{R}_{\geq 0}[P_i]. 
\]
Clearly $\theta \geq \eta$ implies inclusions of the numerical torsion classes, for example, $\overline{\T}_\eta \subseteq \overline{\T}_\theta$.

For a fixed non-zero module $M\in \fmod A$ one can define
\[\mathcal{D}(M)=\{\theta \in \mathbb{R}^l\mid M \text{ is $\theta$-semistable}   \}.\]
The \emph{stability spaces} $\mathcal{D}(M)$ of codimension one are called \emph{walls.}
The open connected components constituting
\[
\mathbb{R}^l \setminus \overline{\bigcup_{M\in \fmod A, M\neq 0}\mathcal{D}(M)}
\]
are called \emph{chambers}. Walls and chambers give the \emph{wall and chamber structure} of $\Ko_0(\proj A) \otimes_\mathbb{Z} \mathbb{R} \simeq \mathbb{R}^l$. Note that for any $M\neq 0$, $\mathcal{D}(M)$ is contained in some wall \cite[Proposition 2.7]{Asai}.

Let us recall how this relates to two-term partial silting complexes. 
An object $T\in \K^b(A)$ is called \emph{presilting} if $\Hom_{\K^b(A)}(T,T[i])=0$ for all $i>0$. It is silting if and only if the smallest thick subcategory containing $T$ is $\K^b(A)$. If $T$ is two-term presilting it is usually called \emph{partial silting}, in this case it can be written as $T=P^{-1} \rightarrow P^0$ with $P^{-1}, P^0$ projective, the $g$-vector of $T$ is $[T]=[P^0]-[P^{-1}]$. A two-term partial silting complex $T\in \K^b(A)$ is silting if and only if it has $l$ non-isomorphic direct summands. 

Let $T=\oplus_{i=1}^t T_i$ be a basic two-term partial silting complex in $\K^b(A)$ decomposed into indecomposable summands. The \emph{cone} $\C(T)$ of $T$ is the cone spanned in $\mathbb{R}^n$ by the $g$-vectors $\{[T_1], \dots , [T_t]\}$: 
\[
\C(T)=\left\{\sum_{i=1}^t \alpha_i [T_i]\mid \alpha_i\geq 0 \right\}.
\]
For example, the cone of $A$ is the first quadrant. Note that for two basic two-term partial silting complexes $T,T'\in \K^b(A)$ we have $\C(T)\cap \C(T')=\C(X)$, where $X$ is a basic two-term partial silting complex such that $\add X=\add T\cap\add T'$, see \cite{DIJ}.  The
 $g$-vectors of indecomposable two-term partial silting objects are rays
of a simplicial polyhedral fan whose maximal cones are $\C(T)$,
where $T\in \K^b(A)$ are two-term silting complexes. This polyhedral fan is called the \emph{$g$-vector fan} of $A$.

By \cite{Bridgeland, BST, Asai},
$\C(T)^\circ$ is a chamber in the wall and chamber structure of $A$ for any two-term silting complex $T$, and any chamber arises like that. The cones of incomplete two-term partial silting complexes are contained in $\bigcup_{M\in \fmod A, M\neq 0}\mathcal{D}(M)$.

The connection between torsion classes, silting objects and t-structures considered in Section \ref{SecSilt} comes from the Happel-Reiten-Smal{\o} tilt. 

For the standard t-structure $(\D^{\leq 0},\D^{> 0})$ in $\D^b(A)$ and a torsion pair $(\T,\F)$ in $\fmod A$, the HRS-tilt of $(\D^{\leq 0},\D^{> 0})$ with respect to $(\T,\F)$ is a new t-structure in $\D^b(A)$ defined as 
\[
(\U,\V)=(\D^{\leq 0}[1]\star \T,\F \star \D^{> 0}).
\]
In case $T\in \K^b(A)$ is a two-term silting complex, the corresponding t-structure $(\U_T,\V_T)$ is $2$-intermediate. This t-structure restricts to the classically considered t-structure $(\U_T\cap \D^b(A),\V_T\cap \D^b(A))$ on $\D^b(A)$. This restricted t-structure can be obtained via HRS-tilt from the standard t-structure with respect to the torsion pair $(\Fac \HH^{0}(T), \HH^0(T)^{\perp})$ in $\fmod A$, where $\Fac$ denotes the subcategory of factor modules of finite coproducts of $M$. The relation between torsion pairs arising from two-term silting complexes and numerical torsion pairs is explained in \cite[Proposition 3.3]{Yurikusa}. Namely, $\Fac \HH^{0}(T)=\overline{\T}_\theta$ for $\theta=[T]$. So, for a two-term silting complex $T\in \K^b(A)$ the t-structure $(\U_T\cap \D^b(A),\V_T\cap \D^b(A))$ on $\D^b(A)$ is given as the HRS-tilt of the standard t-structure with respect to the numerical torsion pair $(\overline{\T}_\theta, \F_\theta)$ where $\theta$ is the $g$-vector of $T$.

In this section we want to study what our construction from Section \ref{SecSilt} gives when we move around $\mathbb{R}^l$ and approach various vectors $\theta$. 
We will use the following corollary from \cite{AsaiIyama}. Note that in \cite{AsaiIyama} it is stated for $\theta^i \in \Ko_0(\K^b(A))\otimes_\mathbb{Z}\mathbb{Q}$, however, it holds in greater generality for $\theta^i \in \Ko_0(\K^b(A))\otimes_\mathbb{Z}\mathbb{R}$.  
\begin{corollary}\label{CorAI}\cite[Corollary 4.6]{AsaiIyama}
    For $\theta \in \Ko_0(\K^b(A))\otimes_\mathbb{Z}\mathbb{R}$, take $\theta^i \in \Ko_0(\K^b(A))\otimes_\mathbb{Z}\mathbb{R}$ for each $i \in \mathbb{N}$ such that $\theta \leq \theta^i$ and $\underset{i\rightarrow \infty}{\mathrm{lim}} \theta^i=\theta$. Then we have 
    \[\overline{\T}_\theta = \bigcap_{i\in\mathbb{N}} \overline{\T}_{\theta^i}.\]
\end{corollary}
The utility of this corollary to our situation is clear since in Section \ref{SecSilt} we described the aisle of a colimit of silting objects as an intersection of aisles. We are ready for the main result of the section which states that the limit of $g$-vectors is compatible with the homotopy colimit of silting objects which answers our initial motivating question. To allow for more freedom we can scale the $g$-vectors of siting complexes by positive multiples, since this does not change the corresponding numerical torsion pair (see Example \ref{Kronecker} for intuition).

\begin{theorem}\label{LimSilt}
    Let $\theta \in \Ko_0(\K^b(A))\otimes_\mathbb{Z}\mathbb{R}$.
    Let $T_i \in \K^b(A)$, $i\geq 0$ be two-term silting complexes such that $[T_i]=\epsilon_i\theta^i$  for some  $\epsilon_i\in \mathbb{R}_{>0}$  and  \[\theta \leq \dots\leq \theta^{i+1}\leq \theta^i\leq \dots \leq \theta^0 \text{ with } \underset{i\rightarrow \infty}{\mathrm{lim}} \theta^i=\theta.\] 
    Let $T\in \D(A)$ be the two-term silting complex constructed from the sequence $T_i$ according to the algorithm described in Theorem \ref{ThmSilt}. Then $\U_{T}\cap \fmod A =\overline{\T}_{\theta}$.
\end{theorem}

\begin{proof}
By assumption the two-term silting complexes $T_i$ satisfy the conditions of the Setup \ref{SiltSet}, as $\theta^{i} \geq \theta^{i+1}$ implies $\overline{\T}_{\theta^{i+1}} \subseteq \overline{\T}_{\theta^{i}}$ implies $\U_{T_{i+1}} \cap \D^b(A) \subseteq \U_{T_{i}}\cap \D^b(A)$ and $T_{i+1}\in \U_{T_i}$. So, Theorem \ref{ThmSilt} is applicable and gives a two-term silting complex $T\in \D(A)$ with $\U_T=\cap_{i\geq 0}\U_{T_i}$.

Note that  $\U_{T_i}\cap \fmod A =\overline{\T}_{\theta^i}$ by \cite[Proposition 3.3]{Yurikusa}. Thus $\U_{T}\cap \fmod A =\cap_{i\geq 0}\overline{\T}_{\theta^i}$. And by Corollary \ref{CorAI}, $\U_{T}\cap \fmod A =\overline{\T}_{\theta}$.  
\end{proof}

Let us now consider the situation where the vector $\theta$ is a $g$-vector of a two-term partial silting complex $U\in\K^b(A)$, so $\theta$ is a part of the wall in a wall and chamber structure of $A$. In this case more can be said about the homotopy colimit $T$.
To the complex $U$ one can associate not one but two torsion pairs in $\fmod A$: 
\[(\Fac \HH^{0}(U), \HH^0(U)^{\perp})\text{ and }  ({}^\perp \HH^{-1}(\nu U), \Sub  \HH^{-1}(\nu U)),\]
where $\nu$ denotes the Nakayama functor and $\Sub M$ denotes the subcategory of submodules of finite coproducts of $M$. Note that
 $\Fac \HH^{0}(U) \subseteq {}^\perp \HH^{-1}(\nu U)$. 
 If $U$ is two-term silting the two torsion pairs coincide.

 In general $U$ has a Bongartz complement $U'$ which is a two-term partial silting complex, such that $U\oplus U'$ is silting. 
The two-term silting complex $U\oplus U'$ has the associated torsion class ${}^\perp \HH^{-1}(\nu U)={}^\perp \HH^{-1}(\nu (U\oplus U'))$ and any other completion $T=U\oplus V$ of $U$ to a two-term silting complex has the property that 
\[\Fac \HH^{0}(U) \subseteq \Fac \HH^{0}(T) \subseteq {}^\perp \HH^{-1}(\nu U)= {}^\perp \HH^{-1}(\nu (U\oplus U'))=\Fac \HH^{0}(U\oplus U').\]

To construct $U'$ one takes the following triangle:
\[
A\rightarrow U' \rightarrow \bar{U} \xrightarrow{f} A[1]
\]
where the map $f$ is an $\add U$-precover of $A[1]$.

Note that the Bongartz complement can be constructed in such a way that $\bar{U}$ is a coproduct of copies of $U$, in which case $[U\oplus U']= [U]+[A]+[\bar{U}]=m[U]+[A]$. So the vector $\mu=[U]+\frac{1}{m}[A]$ is in the chamber of the Bongartz completion $U\oplus U'$ of $U$ and has the property $\mu>\theta=[U]$.

By \cite[Proposition 3.3]{Yurikusa}, $\overline{\T}_\theta={}^\perp \HH^{-1}(\nu U)$ for any $g$-vector in the interior of the cone corresponding to $U$, in particular, for $[U]$.
Combining various considerations we can make the following conclusion:

\begin{corollary}\label{bongartz}
    Let $\theta \in \Ko_0(\K^b(A))\otimes_\mathbb{Z}\mathbb{R}$ be a $g$-vector corresponding to a two-term  partial silting complex $U \in \K^b(A)$.
    Let $T_i \in \K^b(A)$, $i\geq 0$ be two-term silting complex such that \[[T_i]=\epsilon_i\theta^i \text{ for some } \epsilon_i\in \mathbb{R}_{>0} \text{ and }\theta \leq \dots\leq \theta^{i+1}\leq \theta^i\leq \dots \leq \theta^0 \text{ with } \underset{i\rightarrow \infty}{\mathrm{lim}} \theta^i=\theta.\] 
    Let $T\in \D(A)$ be the two-term silting complex constructed from the sequence $T_i$ according to the algorithm described in Theorem \ref{ThmSilt}.
    Then  
    $\Add(T)=\Add(U\oplus U')$, where $U'$ is the Bongartz complement of $U$. 
\end{corollary}

\begin{proof}
By Theorem \ref{LimSilt}, $\U_{T}\cap \fmod A =\overline{\T}_{\theta}$.
On the other hand by \cite[Proposition 3.3]{Yurikusa},  $\overline{\T}_\theta={}^\perp \HH^{-1}(\nu U)$. We get $\overline{\T}_\theta ={}^\perp \HH^{-1}(\nu (U\oplus U'))$. This is the torsion class corresponding to the restriction of the t-structure $(\U_{U\oplus U'},\V_{U\oplus U'})$ to $\D^b(A)$, so $\U_{U\oplus U'}\cap \fmod A =\overline{\T}_{\theta}$ as well.

By \cite[Proposition 6.4]{MZ}, restriction to $\fmod A$ gives a bijection between 2-intermediate silting t-structures in $\D(A)$ and torsion pairs in $\fmod A$, hence the t-structures in $\D(A)$ corresponding to the two silting complexes $T$ and $U\oplus \U'$ coincide. We get $\Add(T)=\Add(U\oplus U')$.
\end{proof}

\begin{remark}
    Note that in the situation of Corollary \ref{bongartz} we can construct a sequence 
    $\theta \leq \dots\leq \theta^{i+1}\leq \theta^i\leq \dots \leq \theta^0$  with $\underset{i\rightarrow \infty}{\mathrm{lim}} \theta^i=\theta$
    such that each $\theta^i$ is a positive multiple of a $g$-vector of a two-term silting complex $T_i$ and each $\theta^i$ belongs to the chamber of the Bongartz completion $U\oplus U'$. For that one can take the vectors $\theta^i=[U]+\frac{1}{i}[A]$, $i>m$. Indeed $\theta^i$ is a positive multiple of $(i-m+1)[U]+(m-1)[U]+[A]=(i-m+1)[U]+[U']$, so it belongs to the chamber of the Bongartz completion. The corresponding two-term silting complex is $T_i=U^{(i-m+1)}\oplus U'$ and $\underset{i\rightarrow \infty}{\mathrm{lim}} \theta^i=\theta$.
\end{remark}

\begin{remark}
     Dual results for the case of two-term cosilting complexes can be obtained replacing $\K^b(\proj A)$ with $\K^b(\inj A)$.
\end{remark}

Let us now restrict our attention to tame algebras. One of our motivations for this section was the following result by Plamondon-Yurikusa.

\begin{theorem}\cite[Theorem 4.1]{PlamondonYurikusa}
    Let $A$ be a finite-dimensional basic algebra over an
algebraically closed field. If $A$ is tame, then its $g$-vector fan is dense in $\mathbb{R}^l$. That means that the closure of the union of the cones of the two-term silting complexes over $A$ is $\mathbb{R}^l$.
\end{theorem}

As the next lemma shows, for a tame algebra $A$ we can approach any vector $\theta$ in $\mathbb{R}^l$ via a sequence of vectors $\theta \leq \dots\leq \theta^{i+1}\leq \theta^i\leq \dots \leq \theta^0$ in $\mathbb{R}^l$ such that each $\theta^i$ is a positive multiple of a silting $g$-vector, which puts us into the context of Theorem \ref{LimSilt}.

\begin{lemma}\label{Lem:g-vect}
    Let $A$ be a basic tame finite-dimensional algebra. For any $\theta$ in $\mathbb{R}^l$ there exists a sequence of vectors $\theta^i$ such that $\theta \leq \dots\leq \theta^{i+1}\leq \theta^i\leq \dots \leq \theta^0$, $\underset{i\rightarrow \infty}{\mathrm{lim}} \theta^i = \theta$, and for each $i$, there exists $\epsilon_i>0$ such that $\epsilon_i \theta_i$ is a $g$-vector of a two-term silting complex in $\K^b(A)$.
\end{lemma}

\begin{proof}
Choose $\theta^0\in \mathbb{Q}^l$ such that $\theta^0 >\theta$ and such that $\epsilon_0\theta^0$ is a $g$-vector of a two-term silting complex for some $\epsilon_0>0$, for example one can take $\theta^0$ to be the vector $\sum_{i=1}^l M[P_i]$ for large enough $M$. For each $\theta^i$ starting with $\theta^0$ consider the set $B_i:=\{ v \in \mathbb{R}^l \mid \theta^{i} > v > \theta\}$. This is an open set by induction. Let us choose $\theta^{i+1} \in B_i \cap B_\theta(\frac{1}{i+1})$, where $B_\theta(c)$ is an open ball of radius $c$ with the center in $\theta$, such that the following condition holds: $\epsilon_{i+1}\theta^{i+1}$ is a $g$-vector of a two-term silting complex for some $\epsilon_{i+1}>0$. This is possible by \cite[Theorem 4.1]{PlamondonYurikusa} since the open set $B_i \cap B_\theta(\frac{1}{i+1})$ must contain an element $\theta^{i+1}$ from the $g$-vector fan of $A$. Note that we can choose $\theta^{i+1}$ to be a vector inside of some chamber of the wall and chamber structure of $A$ by varying $\theta^{i+1}$, since cones of non-complete two-term partial silting complexes are at least of codimension $1$. Similarly we can choose $\theta^{i+1}$ to be in $\mathbb{Q}^l$, since there is a small ball around $\theta^{i+1}$ which is both inside of the chamber of $\theta^{i+1}$ and inside $B_i \cap B_\theta(\frac{1}{i+1})$ and $\mathbb{Q}^l$ is dense in $\mathbb{R}^l$. 
Since $\theta^{i+1}$ is inside of some chamber, there is a two-term silting complex $T_{i+1}$ corresponding to this chamber with indecomposable summand $T_{i+1}^j$. Since $\theta^{i+1}$ has rational entries, it is a combination of $[T_{i+1}^j]$ with positive rational coefficients. Choosing $\epsilon_{i+1}\in\mathbb{Z}_{>0}$ big enough we can guarantee that the coefficients in the decomposition of $\epsilon_{i+1}\theta^{i+1}$ along the basis $[T_{i+1}^j]$ are all positive integers, hence $\epsilon_{i+1}\theta^{i+1}$ is a $g$-vector of a two-term silting complex in $\K^b(A)$. By construction $\underset{i\rightarrow \infty}{\mathrm{lim}} \theta^i = \theta$.
\end{proof}

Using this lemma and Theorem \ref{LimSilt}, for any vector $\theta \in \mathbb{R}^l$, the torsion pair $(\overline{\T}_\theta,\F_\theta)$ can be obtained from a two-term silting complex $T \in \D(A)$, which is a homotopy colimit of two-term silting complexes in $\K^b(A)$.

\begin{corollary}\label{CorTame}
    Let $A$ be a basic tame finite-dimensional algebra. Let $\theta$ be any vector in $\mathbb{R}^l$.
    Let $\theta^i$ be a sequence of vectors constructed in Lemma \ref{Lem:g-vect} and let $T_i$ be the corresponding two-term silting complexes in $\K^b(A)$ with $[T_i]=\epsilon_i\theta^i$. If $T \in \D(A)$ is the two-term silting complex constructed from the sequence $T_i$ according to the algorithm described in Theorem \ref{ThmSilt},
then $\U_{T}\cap \fmod A =\overline{\T}_{\theta}$.
\end{corollary}

\begin{proof}
    This follows from Theorem \ref{LimSilt} and Lemma \ref{Lem:g-vect}.
\end{proof}

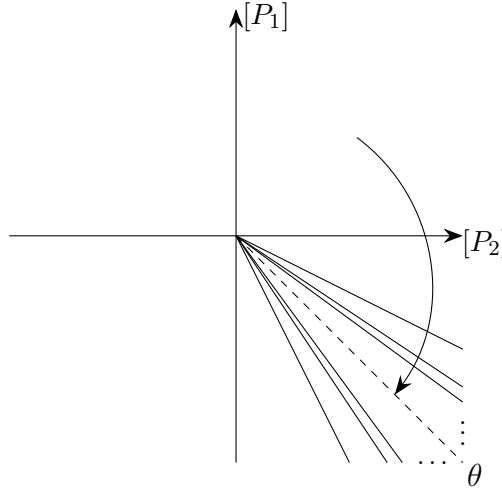
\begin{figure}[!ht]
\centering
\begin{tikzpicture}[scale=0.80]
\tikzstyle{every node}=[font=\fontsize{10.9pt}{15.5pt}\selectfont]
\draw [-{Stealth[scale=1.5]}, ] (6.25,9.75) -- (13.75,9.75);
\draw [-{Stealth[scale=1.5]}, ] (10,6) -- (10,13.5);
\draw [dashed] (10,9.75) -- (13.75,6);
\draw  (10,9.75) -- (13.75,7.875);
\draw  (10,9.75) -- (11.875,6);
\node [font=\fontsize{10.9pt}{15.5pt}\selectfont, inner xsep=0.080cm, inner ysep=0.085cm, rounded corners=0.020cm] at (10.5,13.375) {$[P_1]$};
\node [font=\fontsize{10.9pt}{15.5pt}\selectfont, inner xsep=0.080cm, inner ysep=0.085cm, rounded corners=0.020cm] at (14.125,9.625) {$[P_2]$};
\draw  (10,9.75) -- (13.75,7.25);
\draw  (10,9.75) -- (12.5,6);
\draw  (10,9.75) -- (13.75,7);
\draw  (10,9.75) -- (12.75,6);
\node [font=\fontsize{10.9pt}{15.5pt}\selectfont, inner xsep=0.080cm, inner ysep=0.085cm, rounded corners=0.020cm] at (13.75,6.625) {\small $\vdots$};
\node [font=\fontsize{11.9pt}{15.5pt}\selectfont, inner xsep=0.080cm, inner ysep=0.085cm, rounded corners=0.020cm] at (13.25,6) {\small $\hdots$};

\draw [-{Stealth[scale=1.5]}, ] (12,11.375) arc[start angle=53.37, end angle=-36.63, radius=3.0375];

\node [font=\fontsize{11.9pt}{15.5pt}\selectfont, inner xsep=0.080cm, inner ysep=0.085cm, rounded corners=0.020cm] at (13.95,5.8) { $\theta$};
\end{tikzpicture}
\caption{The wall and chamber structure for the Kronecker quiver.}
\label{fig:my_label}
\end{figure}

\begin{example}\label{Kronecker}
    Let $A=kQ$ be the path algebra of the Kronecker quiver, $Q=1\rightrightarrows 2$. The wall and chamber structure of $A$ is well know and is schematically depicted on Figure \ref{fig:my_label}. In particular, there is a sequence of two-term silting complexes \[T_0=A=P_1\oplus P_2, \mbox{ } T_1=P_2\oplus (P_1\rightarrow P_2^{(2)}), \dots , T_i= (P_1^{(i-1)}\rightarrow P_2^{(i)}) \oplus (P_1^{(i)}\rightarrow P_2^{(i+1)}), \dots  \]
    In the basis $\{[P_2],[P_1]\}$ we have $[T_i]=\binom{2i+1}{-(2i-1)}$. We can take $\theta^i=\binom{(2i+1)/2i}{-(2i-1)/2i}$ and $\epsilon_i=2i$, $i>0$, $\epsilon_0=2$. We get $\underset{i\rightarrow \infty}{\mathrm{lim}} \theta^i = \binom{1}{-1}$. One can think of this process as moving from the first quadrant corresponding to $A$ to the vector $\binom{1}{-1}$ going through the chambers of $T_i$'s as depicted in Figure \ref{fig:my_label}. All these choices place us in Setup \ref{SiltSet}. Thus, by Theorem \ref{ThmSilt} one can construct a diagram in $\Ch(A)$ whose homotopy colimit is a two-term silting complex $T\in\D(A)$ such that $\U_{T}\cap \fmod A =\overline{\T}_{\theta}$. Note that $\overline{\T}_{\theta}$ is the torsion class consisting of all preinjective and regular modules, hence, by \cite[Section 3.4~(3)]{AngeleriInf}, the two-term silting complex $T$ is additively equivalent in $\D(A)$ to the Lukas tilting module~$L$.  
\end{example}

\bibliographystyle{abbrv}
\bibliography{references.bib}

\end{document}